\documentclass{amsart}

\usepackage{float}
\usepackage{hyperref}   
\usepackage{graphicx}   
\usepackage{amssymb,xcolor}
\usepackage{pdfpages}
\usepackage{textcomp}
\usepackage{amsmath}
\usepackage{amsthm}
\usepackage{mathrsfs}
\usepackage{esint}
\usepackage{scalerel}
\usepackage{stackengine}
\usepackage{import}
\usepackage{xifthen}
\usepackage{pdfpages}
\usepackage{transparent}
\usepackage{hyperref}
\usepackage{subcaption}
\usepackage{caption}

\theoremstyle{plain}

\theoremstyle{plain}
\newtheorem{theorem}{Theorem}[section]
\theoremstyle{plain}
\newtheorem{definition}[theorem]{Definition}
\theoremstyle{plain}
\newtheorem{lemma}[theorem]{Lemma}

\theoremstyle{remark}
\newtheorem{remark}[theorem]{Remark}
\theoremstyle{plain}
\newtheorem{proposition}[theorem]{Proposition}
\numberwithin{equation}{section}
\theoremstyle{plain}
\newtheorem{example}[theorem]{Example}
\newtheorem{conjecture}[theorem]{Conjecture}

\makeatletter
\DeclareRobustCommand{\thmprime}{%
  \begingroup
  \expandafter\in@\expandafter b\expandafter{\f@series}%
  \ifin@ \boldmath \fi
  $\m@th{}^{\prime}$%
  \endgroup
}
\makeatother

\stackMath
\newsavebox\tmpbox
\newcommand\widefrown[1]{\ThisStyle{%
\sbox\tmpbox{$\SavedStyle#1$}%
\stackon[0pt]{\usebox{\tmpbox}}{%
\stretchto{%
  \scaleto{%
    \scalerel*[\wd\tmpbox]{\mkern-.8mu\frown\mkern-.8mu}%
    {\rule[-\textheight/2]{1ex}{\textheight}}%
  }{\textheight}%
}{0.8ex}}%
}}

\begin{document}

\title{dihedral rigidity in hyperbolic 3-space}

\author{Xiaoxiang Chai}
\address{Korea Institute for Advanced Study, Seoul 02455, South Korea}
\email{xxchai@kias.re.kr}

\author{Gaoming Wang}
\address{The Chinese University of Hong Kong, Shatin, New Territories, Hong Kong}

\email{gmwang@math.cuhk.edu.hk}

\begin{abstract}
  We prove a comparison theorem for certain types of polyhedra in a 3-manifold
  with its scalar curvature bounded below by $-6$. The result
  confirms in some cases the Gromov dihedral rigidity conjecture in hyperbolic
  $3$-space.
\end{abstract}

\keywords{Dihedral rigidity, hyperbolic space, polyhedron, tetrahedron, mass functional,
constant mean curvature, contact angle, capillary, foliation.}
\subjclass{53C12, 53C21, 53C23, 53C24.}

{\maketitle}

\section{Introduction}
The upper half space model of hyperbolic 3-space $\mathbb{H}^3$ is given by
the metric
\begin{equation}
  b = \tfrac{1}{(x^1)^2} \delta = \tfrac{1}{(x^1)^2} ((\mathrm{d} x^1)^2 +
  (\mathrm{d} x^2)^2 + (\mathrm{d} x^3)^2), \label{metric}
\end{equation}
where $(x^1, x^2, x^3) \in \mathbb{R}^3_+ = \{x \in \mathbb{R}^3 : x^1 > 0\}$.
We represent $\mathbb{H}^3$ as $(\mathbb{R}^3_+, b)$ to emphasize the upper half space model if
needed.

For any vector $\vec{a} = (a^1, a^2, a^3) \in \mathbb{R}^3$ of unit Euclidean
length and any real number $s$, the linear planes (if non-empty)
\begin{equation}
  Z (\vec{a}, s) = \{x \in \mathbb{R}^3 : x^1 > 0, \sum_i a^i x^i =s\}\label{Z}
\end{equation}
give totally umbilic surfaces of $\mathbb{H}^3$ (see Lemma \ref{umbilicity of
Z}). If $a^1 = \pm 1$, $Z (\vec{a}, s)$ is a horosphere; when $a^1 = 0$, $Z
(\vec{a}, s)$ is totally geodesic; and $0 < |a^1 | < 1$, $Z (\vec{a}, s)$ is
called an equidistant surface. Given any two different $(\vec{a}_1, s_1)$ and
$(\vec{a}_2, s_2)$, the two linear planes $Z (\vec{a}_1, s_1)$ and $Z
(\vec{a}_2, s_2)$ intersect at a constant angle if their intersection is
nonempty since $b$ is conformal to the Euclidean metric.

We consider polyhedra in the hyperbolic 3-space enclosed by some $Z (\vec{a},
s)$. In hyperbolic geometry, a hyperbolic polyhedron is
usually enclosed by totally geodesic surfaces. It turns out that in scalar
curvature geometry, the polyhedron enclosed by $Z (\vec{a}, s)$ seems more
natural.

First, we fix some conventions. Let $M$ be a manifold with boundary (and corners) diffeomorphic to some polyhedron $\bar{P}\subset \mathbb{R}^3$. Let $\Sigma$ be a surface in $M$, and $N$ be a chosen unit normal of the surface $\Sigma$, the second fundamental form in the direction of $N$ is given by $A(X,Y)=\langle \nabla_{X}N, Y\rangle$. The mean curvature of $\Sigma$ in the direction of $N$ is just the trace of $A$. In this convention, the horosphere $\{x^1=1\}$ has mean curvature $-2$ in the direction of the unit normal $x^1\tfrac{\partial}{\partial x^1}$ in the hyperbolic space. Let $F_i$ and $F_j$ be two neighboring faces of $M$, $X_i$ and $X_j$ are the corresponding unit normal pointing to the outside of $M$, the \text{{\itshape{dihedral angle}}} $\measuredangle_{ij}M$ of two neighboring faces  $F_i$ and $F_j$ is define to be
\begin{equation}
\cos\measuredangle_{ij}M = -\langle X_i, X_j \rangle.\label{eq:dihedral angle}
\end{equation}
In other words, the dihedral angle is the complementary angle of the angle formed by the outward unit normals. We also use the notation $\measuredangle F_iF_j$.

Motivated by the result regarding the hyperbolic mass in \cite{jang-hyperbolic-2021}, Appendix \ref{eval}, and the dihedral rigidity conjectures \cite{gromov-dirac-2014}, \cite{li-polyhedron-2020}, \cite{li-dihedral-2020} and \cite{li-dihedral-2020-1},
we conjecture the following.

\begin{conjecture}
  \label{gromov conjecture hyperbolic}Let $\bar{P}$ be a hyperbolic reference polyhedron, that is, $\bar{P}$ is enclosed by some
  $Z (\vec{a}, s)$ in the hyperbolic $3$-space $(\mathbb{R}_+^3, b)$. Let the faces of $\bar{P}$ be
  faces $\bar{F}_i$. Assume that $\bar{P}$ is convex in $(\mathbb{R}^3_+, \delta)$, let
  $(M^3, g)$ be a Riemannian manifold diffeomorphic to $\bar{P}$, if
  \begin{enumerate}
    \item the scalar curvature of $M$ satisfies $R_g \geq - 6$;
    
    \item the mean curvatures of the faces of $M$ are no less than the mean
    curvatures of the corresponding faces of $\bar{P}$;
    
    \item and the dihedral angles along the edges of $M$ are no greater than
    the dihedral angles along the edges of $\bar{P}$ i.e. $\measuredangle_{ij}M \leq \measuredangle_{ij} \bar{P}$,
  \end{enumerate}
  then $M$ is isometric to $\bar{P}$. Here, the mean curvature is computed in the direction of the unit normal pointing outward of $M$ or $\bar{P}$. 
\end{conjecture}

It is easy to make the conjecture in other dimensions. Gromov \cite{gromov-dirac-2014} proposed the dihedral rigidity conjectures as an attempt to understand scalar curvature with a lower bound for non-Riemannian metric spaces.
As triangle comparison theorems are of great importance in the study of sectional curvature for non-Riemannian metric spaces, Gromov suggested that the convex polyhedron in the Euclidean space could play a similar role in the study of nonnegative scalar curvature. So in \cite{gromov-dirac-2014} he proposed the Euclidean version of Conjecture \ref{gromov conjecture hyperbolic} which was confirmed by Li \cite{li-polyhedron-2020}, \cite{li-dihedral-2020} in some cases. Recently there was a spinorial proof \cite{wang-gromovs-2022}. 

Gromov also made the conjecture for parabolic cubes in \cite{gromov-dirac-2014}, and it was naturally important in the study of scalar curvature with a negative lower bound.
Li \cite{li-dihedral-2020-1} confirmed Gromov dihedral rigidity conjecture for parabolic cubes in dimensions up to seven.

Another motivation for the dihedral rigidity comes from the positive mass theorem. The positive mass theorem for asymptotically flat manifolds says that if the scalar curvature of the manifold is nonnegative, then the ADM mass is nonnegative (see \cite{schoen-proof-1979}). Lohkamp \cite{lohkamp-scalar-1999} argued that if the ADM mass is negative, then he can find a new metric on the asymptotically flat manifold preserving the nonnegativity of the scalar curvature. And the new metric has the extra property that the scalar curvature is positive somewhere and flat outside a big compact set. By identifying the opposite edges of a big cube, he reduced the positive mass theorem to the non-existence of positive scalar curvature metrics on the torus.

 Instead of chopping away the region outside the cubes, we chose to chop away the region outside a convex polyhedron of the Euclidean space and now we can view the dihedral rigidity for the remaining region as a localization of the positive mass theorem. In particular, see the evaluation results \cite{miao-measuring-2020,miao-mass-2021} which localize the ADM mass on polyhedra.

For the hyperbolic case, following the same philosophy as the Euclidean case, given a manifold which is hyperbolic outside a compact set, we chop away the infinity along the boundaries made up by pieces of $Z(\vec{a},s)$. We can study the dihedral rigidity for the remaining region. In Conjecture \ref{gromov conjecture hyperbolic}, we restate Gromov dihedral rigidity conjecture in the hyperbolic space to include more general polyhedra. As mentioned earlier, our motivation partly comes from the localization of the hyperbolic mass of Jang-Miao \cite{jang-hyperbolic-2021} and Appendix \ref{eval}, and the relevant hyperbolic positive mass theorems which can be found for example in \cite{wang-mass-2001,chrusciel-mass-2003,andersson-rigidity-2008}.

We are able to settle Conjecture \ref{gromov conjecture hyperbolic} for some special polyhedra,
and our result is an analog of Li \ {\cite{li-polyhedron-2020}} who showed the
dihedral rigidity for some polyhedra in the Euclidean 3-space.

Now we define two types of polyhedra. Let $k \geq 3$ be an integer, $\vec{a}$ be a constant vector in $\mathbb{R}^3$ of unit Euclidean length and $s_1$, $s_2$ be two numbers with $0<s_1<s_2$.

\begin{definition}
  \label{cone}In $\mathbb{R}^3$, let $\bar{B} \subset \{x : \vec{a} \cdot x = s_1
  \}$ be a convex $k$-polygon and $\bar{p} \in \{x : x \cdot \vec{a} = s_2 \}$ be a
  point. We call the set
  \begin{equation*}
    \bar{P} := \{t \bar{p} + (1 - t) x : t \in [0, 1], x \in \bar{B}\}
  \end{equation*}
  a $(\bar{B}, \bar{p})$-cone, $\bar{B}$ the base face and all the other faces side faces. We
  call a Riemannian manifold $(M, g)$ a cone type polyhedron if it is
  diffeomorphic to a $(\bar{B}, \bar{p})$-cone.
\end{definition}

\begin{definition}
  \label{prism}In $\mathbb{R}^3$, let $\bar{B}_1 \subset \{x : \vec{a} \cdot x = s_1
  \}$ and $\bar{B}_2 \subset \{x : \vec{a} \cdot x = s_2 \}$ be two convex
  $k$-polygons whose corresponding edges are parallel.
  We call the set
  \begin{equation*}
    P := \{t p + (1 - t) q : t \in [0, 1], p \in \bar{B}_1, q \in \bar{B}_2 \}
  \end{equation*}
  a $(\bar{B}_1, \bar{B}_2)$-prism, $\bar{B}_1$ the bottom face, $\bar{B}_2$ the top face and all
  other faces side faces. We call a Riemannian manifold $(M, g)$ a prism 
  polyhedron if it is diffeomorphic to a $(\bar{B}_1, \bar{B}_2)$-prism.
\end{definition}

\begin{figure}[H]
    \centering
	\begingroup
	\fontsize{9pt}{12pt}
	\def\svgwidth{0.8\columnwidth}
	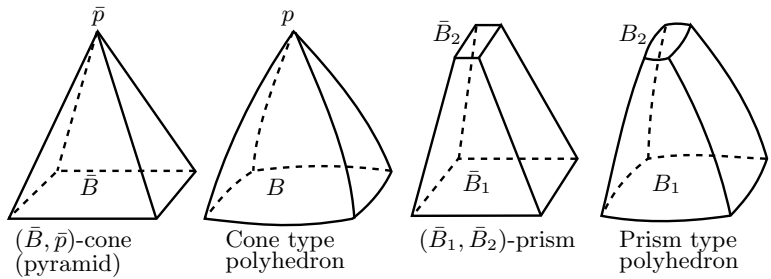
	\endgroup

    \caption{Two types of polyhedrons}
    \label{fig:two-types-of-polyhedrons}
\end{figure}

We call $\bar{P}$ with the metric \eqref{metric} in Definitions \ref{cone} and
\ref{prism} a \text{{\itshape{reference polyhedron}}} or
\text{{\itshape{hyperbolic reference polyhedron}}}, that is, $\bar{P}$ is enclosed by
some $Z (\vec{a}, s)$ in $(\mathbb{R}^3_+, b)$.
We add a bar for every geometric quantity related to reference polyhedra.
Sometimes we need to study $\bar{P}$
with the Euclidean metric, to emphasize the Euclidean background, we call $\bar{P}$
a \text{{\itshape{flat reference}}}. 
We have mentioned the parabolic cube in the work of \cite{gromov-dirac-2014,li-dihedral-2020-1}, and a parabolic cube is a cube with two base faces lying on the horospheres.

We only consider such the hyperbolic reference polyhedron $\bar{P}$ whose base face $\bar{B}$ lies on
the horosphere $\{x^1 = s_1 \}$ and the other base face or the vertex of the
polyhedron lies on the horosphere $\{x^1 = s_2 \}$ with $s_2 > s_1 > 0$.
Also, due to technical issues (see Conjecture \ref{conj_rig_pyramid}), we only treat the tetrahedron case among all cone type polyhedra. 

\begin{theorem}
  \label{dihedral non-rigidity}Let $(M^3, g)$ be a Riemannian polyhedron with
  side faces $F_1$, $\cdots$, $F_k$, assume that $(M, g)$ is diffeomorphic to
  $\bar{P} \subset (\mathbb{R}^3_+, b)$ which is a tetrahedron or prism with side faces $\bar{F}_1$, $\cdots$, $\bar{F}_k$. Denote $\bar{\gamma }_j$ the angle between
  $\bar{F}_j$ and the base face of $\bar{P}$ (if $\bar{P}$ is a prism, fix one base face with
  the smaller $x^1$ coordinate). Assume that everywhere along $F_j \cap F_{j +
  1}$,
  \begin{equation}\label{extra dihedral angle condition}
    | \pi - (\bar{\gamma }_j + \bar{\gamma }_{j + 1}) | < \measuredangle F_j F_{j + 1},
  \end{equation}
  and the base face of $\bar{P}$ lies on one of the horospheres, and the scalar
  curvature of $M$ is greater than $- 6$ and the mean curvature $H$ of each
  face of $M$ are greater than those of $\bar{P}$ in $\mathbb{H}^3$, then the
  dihedral angles cannot be everywhere strictly less than those of $\bar{P}$.
\end{theorem}

To state the rigidity part, we introduce the cylinder trapping condition.

\begin{definition}[Cylinder trapping condition]
  \label{cylinder condition}We say that flat reference $\bar{P}$ is a cone or prism type polyhedron
  satisfies the cylinder trapping condition if there exists an infinite solid cylinder
  $\bar{\mathcal{C}}$ such that $\bar{P} \subset \bar{\mathcal{C}}$, the vertices of one base
  face $\bar{B}$ lie in the edges of $\bar{\mathcal{C}}$, $\partial \bar{B} \subset \partial
  \bar{\mathcal{C}}$, and in the case of that $\bar{P}$ is of prism type at least one edge of the other face lies strictly in $\bar{\mathcal{C}}$.
\end{definition}

\begin{figure}[H]
    \centering
	\begingroup
	\fontsize{9pt}{12pt}
	\def\svgwidth{0.8\columnwidth}
\begingroup%
  \makeatletter%
  \providecommand\color[2][]{%
    \errmessage{(Inkscape) Color is used for the text in Inkscape, but the package 'color.sty' is not loaded}%
    \renewcommand\color[2][]{}%
  }%
  \providecommand\transparent[1]{%
    \errmessage{(Inkscape) Transparency is used (non-zero) for the text in Inkscape, but the package 'transparent.sty' is not loaded}%
    \renewcommand\transparent[1]{}%
  }%
  \providecommand\rotatebox[2]{#2}%
  \newcommand*\fsize{\dimexpr\f@size pt\relax}%
  \newcommand*\lineheight[1]{\fontsize{\fsize}{#1\fsize}\selectfont}%
  \ifx\svgwidth\undefined%
    \setlength{\unitlength}{680.31496063bp}%
    \ifx\svgscale\undefined%
      \relax%
    \else%
      \setlength{\unitlength}{\unitlength * \real{\svgscale}}%
    \fi%
  \else%
    \setlength{\unitlength}{\svgwidth}%
  \fi%
  \global\let\svgwidth\undefined%
  \global\let\svgscale\undefined%
  \makeatother%
  \begin{picture}(1,0.33333333)%
    \lineheight{1}%
    \setlength\tabcolsep{0pt}%
    \put(0,0){\includegraphics[width=\unitlength,page=1]{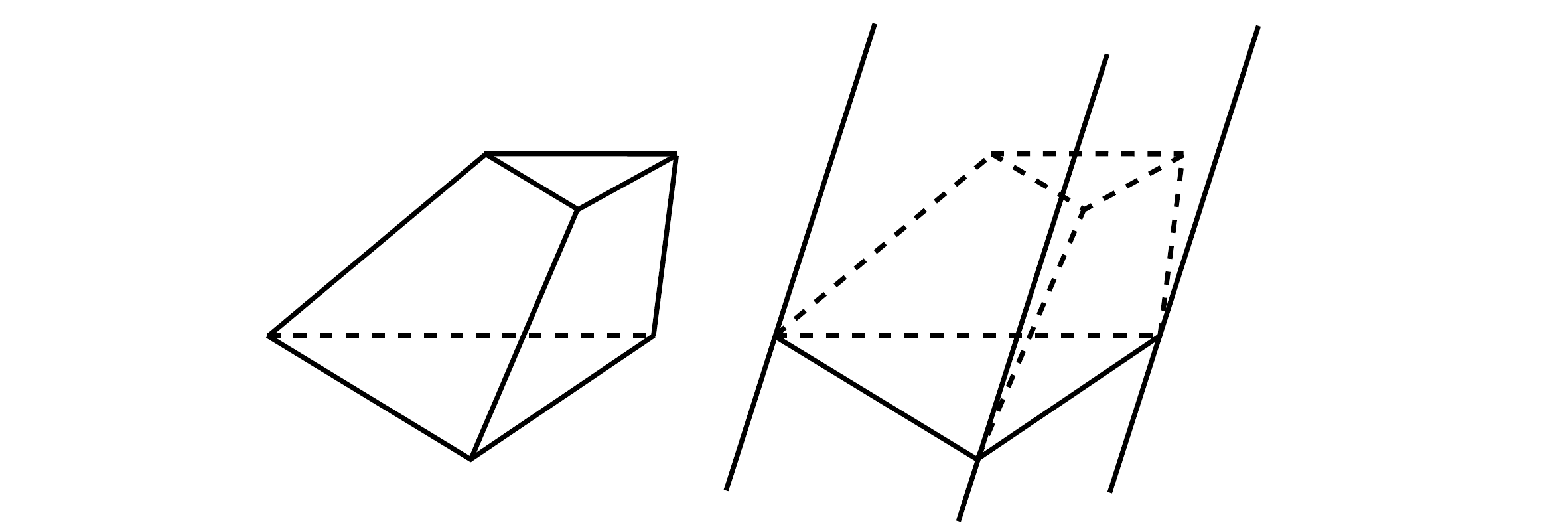}}%
    \put(0.31458761,0.25020549){\color[rgb]{0,0,0}\makebox(0,0)[lt]{\lineheight{0}\smash{\begin{tabular}[t]{l}A prism $\bar{P}$\end{tabular}}}}%
    \put(0.56238802,0.29495239){\color[rgb]{0,0,0}\makebox(0,0)[lt]{\lineheight{0}\smash{\begin{tabular}[t]{l}$\bar{\mathcal{C}}$\end{tabular}}}}%
    \put(0.48063045,0.03632562){\color[rgb]{0,0,0}\makebox(0,0)[lt]{\lineheight{0}\smash{\begin{tabular}[t]{l}$\bar{P}\subset \bar{\mathcal{C}}$\end{tabular}}}}%
    \put(0,0){\includegraphics[width=\unitlength,page=2]{fig-cylinder-trapping.pdf}}%
  \end{picture}%
\endgroup%

	\endgroup

    \caption{A prism satisfying the cylinder trapping condition.}
    \label{fig:fig-cylinder-trapping}
\end{figure}
See Figure \ref{fig:fig-cylinder-trapping} for a flat prism type polyhedron satisfying the cylinder trapping condition.
We then have the rigidity statement.

\begin{theorem}
  \label{rigidity}Under the same assumptions of Theorem \ref{dihedral
  non-rigidity} and the extra assumptions that the reference polyhedron $\bar{P}$
  satisfies the cylinder trapping condition or
  \begin{equation}
  \bar{\gamma }_j\leq \tfrac{\pi}{2},\text{ for all } j=1,2,\cdots,k \label{li foliation restriction}
  \end{equation}
  we have the rigidity statement. Namely, if $R_g \geq - 6$, the mean
  curvature of each face of $M$ is no less than those $\bar{P}$ and the dihedral
  angles $\measuredangle_{i j} M \leq \measuredangle_{i j} \bar{P}$, then $(M,
  g)$ is isometric to a polyhedron in hyperbolic 3-space.
\end{theorem}

\begin{remark}
    Note that any cone type polyhedron satisfies the cylinder trapping condition (see Lemma \ref{lem_tetrahedron_trapping} for tetrahedron),
    and a flat prism type polyhedron $\bar{P}$ satisfies cylinder trapping condition if and only if one of the base faces can be put inside another base face strictly.
    Hence, Theorem \ref{rigidity} holds for any possible tetrahedron $\bar{P}$.
\end{remark}

    Because we have fixed the base face of the flat prism to be the one with the smaller $x^1$ coordinates i.e. $\bar{B}_1$, the cylinder trapping condition means that the face $\bar{B}_2$ with the greater $x^1$ coordinate is smaller in $(\mathbb{R}^3_+,\delta)$. If the base face is smaller, we need to switch the roles of $\bar{B}_1$ and $\bar{B}_2$ and do the corresponding modifications to the functional \eqref{action}. Instead of studying stable surfaces of mean curvature $-2$, we need to study surfaces of mean curvature $2$. We leave these modifications and corresponding versions of Theorems \ref{dihedral non-rigidity} and \ref{rigidity} to the readers.

    We would like to mention that our work has a few differences to Li \cite{li-polyhedron-2020}:
    (I) We introduce the cylinder trapping condition (Definition \ref{cylinder condition}) which allows us to include a wider class of polyhedra in the rigidity Theorem \ref{rigidity}; (II)
    Whether Theorem \ref{dihedral non-rigidity} holds for general cone type polyhedra is quite delicate, the original proof of Li contains a delicate mistake that he was only able to correct under additional assumptions for general cone type polyhedra (see Conjecture \ref{conj_rig_pyramid}). We can only remove his assumptions for tetrahedra (see Proposition \ref{prop_energy_est}).
    (III) Using our proof for the CMC foliation near the vertex, one can remove a restrictive assumption Li implicitly made on the metric near the vertex in the Euclidean dihedral rigidity. (Compare Section \ref{sec:foliation near vertex} and \cite[Section 4.2]{li-polyhedron-2020}.)

The article is organized as follows:

In Section \ref{existence and regularity}, we formulate the variational problem which leads to the proof of Theorem \ref{dihedral non-rigidity} and \ref{rigidity}. In Section \ref{cone type polyhedron}, we analyse a closely related functional \eqref{functional in flat background} on the flat cone type polyhedron in detail. In Section \ref{sec:dihedral non-rigid}, we give the proof of Theorem \ref{dihedral non-rigidity}. In the Section \ref{sec:rigidity}, we conclude the rigidity of Theorem \ref{rigidity}.

The Appendix \ref{variation of mean curvature and contact angle} contains the variational formulas for the mean curvature and the contact angle. In Appendix \ref{eval}, we evaluate a component of the hyperbolic mass on a family of polyhedra. In Appendix \ref{proof of minimiser in general tetrahedron}, we give the proof of Proposition \ref{prop_energy_est}.

\

\section*{{\textbf{Acknowledgements}}}

We thank Prof. Chao Li (NYU) for his detailed explanation of his paper \cite{li-polyhedron-2020} and useful comments on our work.  We also thank our advisor Prof. Martin Li (CUHK) for his constant support and encouragement.  X. Chai would like to acknowledge the support from the KIAS (Korea Institute for Advanced Study) research fellow grant MG074402 and thank Inkang Kim (KIAS) for explaining some basics of tetrahedra in hyperbolic space related to the paper \cite{chai-asymptotically-2021} and the Appendix \ref{eval}. G. Wang is partially supported by a research grant from the Research Grants Council of the Hong Kong Special Administrative Region, China [Project No.:CUHK 14304120].

\

\section{Existence and regularity}\label{existence and regularity}

We indicate how to prove Theorem \ref{dihedral non-rigidity}. In the upper half space model, a polyhedron as in Definitions \ref{cone} and \ref{prism} is foliated by stable surfaces of constant mean curvature $\pm 2$ and each leaf forms a constant contact angle with a side face. In a Riemannian polyhedron satisfying the conditions in Theorem \ref{dihedral non-rigidity}, we wish to find a surface sharing the same properties with the leaf. After such a surface is found, an analysis of the stability inequality would finish the proof. This is the goal of this section.

\subsection{The functional}\label{the functional}

We define the energy functional,
\[
    \mathcal{F}(E)=\mathcal{H}^2 (\partial E \cap \mathring{M}) - 2\mathcal{H}^3 (E) - \sum_j \cos
  \bar{\gamma}_j \mathcal{H}^2 (\partial E \cap F_j),\quad \text{for }E\in \mathscr{E}, \label{action in terms of cacciopolli}
\]
where $\mathscr{E}$ is defined as
$$
    \mathscr{E}:=
    \begin{cases}
        \{E\subset M: p \in E, E\cap B=\emptyset\}, & P\text{ is of cone type},\\
        \{E\subset M: B_1 \subset E, E\cap B_2=\emptyset\}, & P\text{ is of prism type}.
    \end{cases}
$$

Now, we consider the variational problem
\begin{equation}
I=\inf \{\mathcal{F}(E):E\in \mathscr{E}\}.
\label{action}
\end{equation}

If it happens $\Sigma:=\partial E \cap \mathring{E}$ is a $C^{1,\alpha}$ surface, we use the notations as shown in Figure \ref{labelling-of-vectors}.


\begin{figure}[ht]
    \centering
    %
	\begingroup
	\fontsize{9pt}{12pt}
	\def\svgwidth{0.8\columnwidth}
	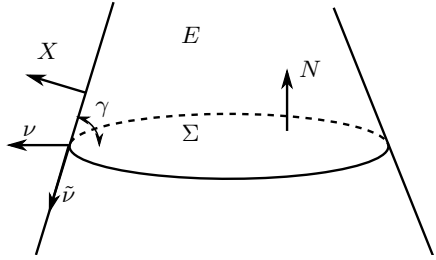
	\endgroup

    \caption{Labelling of vector names.}
    \label{labelling-of-vectors}
\end{figure}

We list these notations as follows:
\begin{itemize}
    \item $X$, the outward pointing unit normal vector field of $\partial M$ in $M$.
    \item $N$, the unit normal vector field of $\Sigma$, pointing into $E$.
    \item $\tilde{\nu}$, the unit normal vector field of $\partial \Sigma$ in $\partial M$, pointing outward $E$.
    \item $A$ (resp. ${A_{\partial M}}$), the second fundamental form of $\Sigma$ (resp. $\partial M$) in $M$.
    \item $H$ (resp. ${H_{\partial M}}$), the mean curvature of $\Sigma$ (resp. $\partial M$) in $M$.
\end{itemize}

Let $\phi_t$ be a family of diffeomorphisms $\phi_t : \Sigma \to M$ such that
$\phi_t (\partial \Sigma) \subset \partial M$ and $\phi_0 (\Sigma) = \Sigma$.
Let $\Sigma_t$ be $\phi_t (\Sigma)$ and $E_t$ be the corresponding component
separated by $\Sigma_t$. Let $Y$ be the vector field
$\tfrac{\partial}{\partial t} \phi_t$, and $Y$ is tangent to $\partial M$.

Define $\mathcal{A}(t)$ as
\begin{equation*}
  \mathcal{A} (t) =\mathcal{H}^2 (\Sigma_t) - 2\mathcal{H}^3 (E_t) - \sum_j
  \cos \bar{\gamma }_j \mathcal{H}^2 (\partial E_t \cap F_j).
\end{equation*}
By the first variation formula, let $f = \langle Y, N \rangle$, and
\begin{equation}
  \mathcal{A}' (0) = \int_{\Sigma} (H + 2) f + \sum_j \int_{\partial \Sigma
  \cap F_j} \langle Y, \nu - \cos \bar{\gamma }_j \tilde{\nu} \rangle .
  \label{first variation of action}
\end{equation}
We get $2 \int_{\Sigma} f$ from the variation of $- 2\mathcal{H}^3 (E_t)$
because $N$ points inward of $E_t$ and $\Sigma$ moving in the direction of $N$
would decrease $\mathcal{H}^3 (E_t)$.

Assume that $\Sigma$ is of constant mean curvature $-2$, we then have the second
variational formula:
\begin{equation}
   \mathcal{A}'' (0) = - \int_{\Sigma}[ f \Delta f +
  (|A|^2 +\ensuremath{\operatorname{Ric}}(N)) f^2 ]+ \sum_{j }
  \int_{\partial \Sigma \cap F_j} f (\tfrac{\partial f}{\partial \nu} - Q f),
  \label{second variation}
\end{equation}
where $Q$ is given on $\partial \Sigma \cap F_j$ by
\begin{equation}
  Q f = \tfrac{1}{\sin \bar{\gamma }_j} {A_{\partial M}} (\tilde{\nu}, \tilde{\nu}) - \cot \bar{\gamma}_j A
  (\nu, \nu) . \label{Q}
\end{equation}
This formula can be found in \cite{ros-stability-1997}, but notice that our convention of the second fundamental form $A$ differs from \cite{ros-stability-1997} by a sign. See also Appendix \ref{variation of mean curvature and contact angle}.

\subsection{Existence and regularity}

The existence and regularity of the minimiser for the variational problem (\ref{action})
is an easy extension of Li's work. We have the following.

\begin{theorem} [{\cite[Theorem 2.1]{li-polyhedron-2020}}]
\label{thm_regularity}
Let $(M^3,g)$ be a Riemannian polyhedron of cone or prism type.
Suppose the mean curvature of each face of $M$ is no less than those $\bar{P}$, dihedral angles $\measuredangle_{ij}M\le
\measuredangle_{ij}\bar{P}$, and $M$ satisfies angle condition \eqref{extra dihedral angle condition}.

Let $E\in \mathscr{E}$ be a minimiser for the variational problem (\ref{action}).
Suppose $E\neq \emptyset$ if $\bar{P}$ is of cone type.

Then $\Sigma=\partial E \cap \mathring{M}$ has constant mean curvature $-2$, $C^{1,\alpha}$ to its corners for some $\alpha\in (0,1)$, and meets $F_j$ at constant angle $\bar{\gamma}_j$.
\end{theorem}

The proof for Theorem \ref{thm_regularity} is essentially the same as Li's proof.

We list the key steps in this proof here.
The interested reader is referred to \cite{li-polyhedron-2020}.

\begin{enumerate}
    \item By using the maximum principle (see \cite{SolomonWhite-max-principle-1989}) and boundary regularity theorem (see \cite{Taylor-regularity-1977, Philippis-regularity-2014}), we know $\Sigma$ is regular in the interior and on the boundary away from the corner.
    \item Establish regularity near the corners by adapting Simon's argument \cite{Simon-corner-regular-1993} and using curvature estimates.  
\end{enumerate}

\begin{remark}
    Note that recently, Zhang \cite{Zhang-max-principle-2022} gave a boundary maximum principle for stationary varifolds with fixed contact angle conditions.
    This result can help us shorten the proof of Theorem \ref{thm_regularity}.
\end{remark}

\section{Cone type polyhedron}\label{cone type polyhedron}

In this section, we show that if the Riemannian polyhedron $(M,g)$ of tetrahedron type has dihedral angles less than its reference along every edge, then the inifimum of the energy \eqref{action in terms of cacciopolli} is negative. 

The energy evaluated on a family of surfaces approaching the apex $p$ of $(M,g)$ will limit to zero. Because of the scaling, the term $\mathcal{H}^3(E)$ would play less role in determining the sign of \eqref{action in terms of cacciopolli} as surfaces in a tiny neighbourhood of $p$ approach $p$. Even better, we only have to consider \eqref{functional in flat background}. If we find a minimiser which makes \eqref{functional in flat background} negative when the reference is a tetrahedron (see Propositions \ref{prop_energy_est} and \ref{prop_pyramid_another_condition}), then this in turn would guarantee the existence of a surface sufficiently close to the apex $p\in M$ with \eqref{action in terms of cacciopolli} negative.

Meanwhile, Example \ref{tetrahedron example} showed that there are polyhedral cones with the same dihedral angles but not isometric to each other. Together with Conjecture \ref{minimiser existence conjecture}, it made Theorems \ref{dihedral non-rigidity} and \ref{rigidity} quite difficult for cone type polyhedra other than tetrahedra.

To describe the geometric properties of cone type polyhedra, first, 
we recall some basic facts about polyhedral cones.
Note that several new notations are only used in this section.

\subsection{Basic properties of convex polyhedral cones}%
\label{sec:basic_properties_of_convex_polyhedral_cones}

In this subsection, we give some basic results about convex polyhedral cones.
Since these results are quite standard, we omit the proof.
Interested readers are referred to \cite[Chapter 1]{fenchel-cone-1953} for example.

\begin{definition}
	Let $\left\{ u_1,\cdots ,u_k \right\}\subset \mathbb{R}^3 $ be the set of non-zero points in $\mathbb{R}^3 $.
	We say $\boldsymbol{C}$ is a \textit{(convex) polyhedral cone} spanned by $\left\{ u_1,\cdots ,u_k \right\}$ if
	\[
		\boldsymbol{C}=\left\{ \sum_{j =1}^{n}\lambda_j u_j: \lambda_j\ge 0, \forall 1\le j\le k \right\}.
	\]
	We write such cone as $\boldsymbol{C}=\mathrm{span}(u_1,\cdots ,u_k)$.
	\label{def_convex_poly_cone}
\end{definition}

\begin{figure}[ht]
    \centering
	\begingroup
	\fontsize{9pt}{12pt}
	\def\svgwidth{0.8\columnwidth}
\begingroup%
  \makeatletter%
  \providecommand\color[2][]{%
    \errmessage{(Inkscape) Color is used for the text in Inkscape, but the package 'color.sty' is not loaded}%
    \renewcommand\color[2][]{}%
  }%
  \providecommand\transparent[1]{%
    \errmessage{(Inkscape) Transparency is used (non-zero) for the text in Inkscape, but the package 'transparent.sty' is not loaded}%
    \renewcommand\transparent[1]{}%
  }%
  \providecommand\rotatebox[2]{#2}%
  \newcommand*\fsize{\dimexpr\f@size pt\relax}%
  \newcommand*\lineheight[1]{\fontsize{\fsize}{#1\fsize}\selectfont}%
  \ifx\svgwidth\undefined%
    \setlength{\unitlength}{680.31496063bp}%
    \ifx\svgscale\undefined%
      \relax%
    \else%
      \setlength{\unitlength}{\unitlength * \real{\svgscale}}%
    \fi%
  \else%
    \setlength{\unitlength}{\svgwidth}%
  \fi%
  \global\let\svgwidth\undefined%
  \global\let\svgscale\undefined%
  \makeatother%
  \begin{picture}(1,0.375)%
    \lineheight{1}%
    \setlength\tabcolsep{0pt}%
    \put(0,0){\includegraphics[width=\unitlength,page=1]{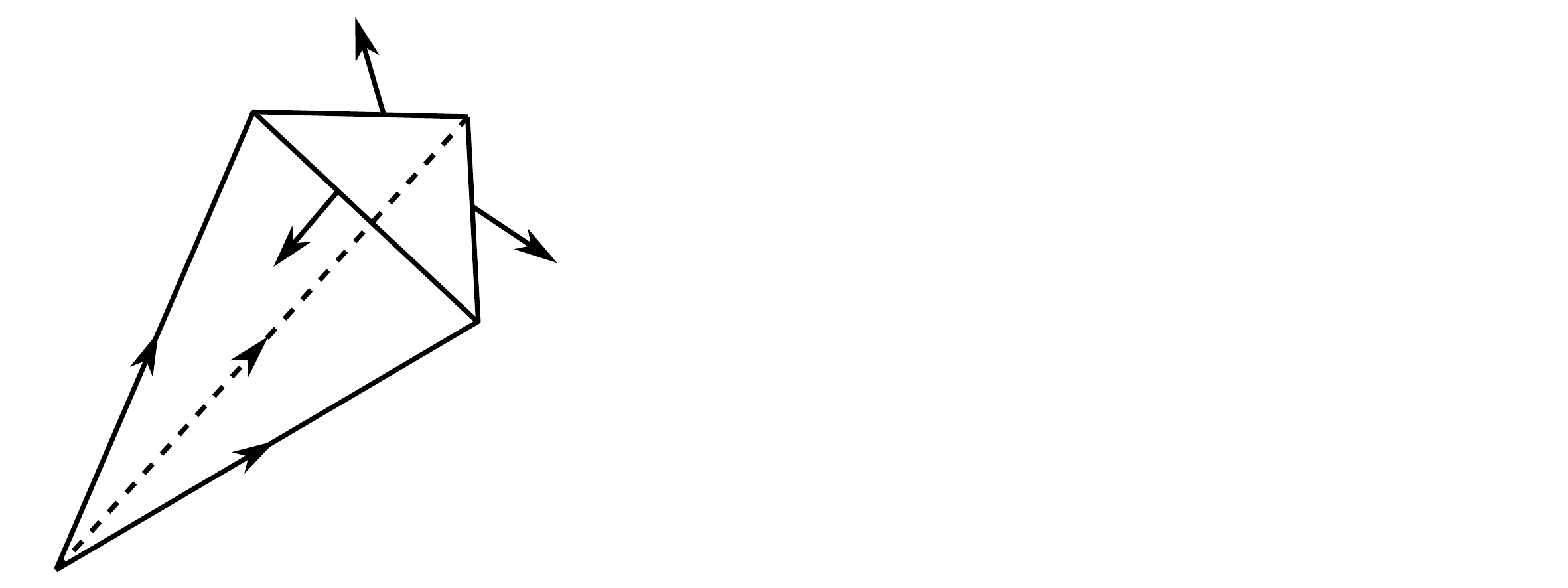}}%
    \put(0.0972812,0.02069186){\color[rgb]{0,0,0}\makebox(0,0)[lt]{\lineheight{0}\smash{\begin{tabular}[t]{l}$u_i$, span vectors\end{tabular}}}}%
    \put(0.244729,0.3449715){\color[rgb]{0,0,0}\makebox(0,0)[lt]{\lineheight{0}\smash{\begin{tabular}[t]{l}$n_i$, face vectors\end{tabular}}}}%
    \put(0,0){\includegraphics[width=\unitlength,page=2]{cones.pdf}}%
    \put(0.68141797,0.07778955){\color[rgb]{0,0,0}\makebox(0,0)[lt]{\lineheight{0}\smash{\begin{tabular}[t]{l}$\boldsymbol{C}$\end{tabular}}}}%
    \put(0.68286545,0.27098364){\color[rgb]{0,0,0}\makebox(0,0)[lt]{\lineheight{0}\smash{\begin{tabular}[t]{l}$\boldsymbol{C}^*$\end{tabular}}}}%
    \put(0.54535439,0.09587888){\color[rgb]{0,0,0}\makebox(0,0)[lt]{\lineheight{0}\smash{\begin{tabular}[t]{l}$u_i$\end{tabular}}}}%
    \put(0.64833328,0.22148378){\color[rgb]{0,0,0}\makebox(0,0)[lt]{\lineheight{0}\smash{\begin{tabular}[t]{l}$n_i$\end{tabular}}}}%
  \end{picture}%
\endgroup%

	\endgroup

    \caption{A polyhedral cone and its polar cone}
    \label{fig:cones}
\end{figure}

There is another way to define the polyhedral cones by the intersection of half-spaces.
\begin{definition}
	Let $\left\{ n_1,\cdots ,n_k \right\}\subset \mathbb{R}^3 $ be the set of non-zero vectors in $\mathbb{R}^3 $.
	We say $\boldsymbol{C}$ is a \textit{polyhedral cone} generated by $\left\{ n_1,\cdots ,n_k \right\}$ if
	\[
		\boldsymbol{C}=\left\{ x \in \mathbb{R}^3 :x \cdot n_j \le 0 , \forall 1\le j\le k\right\}.
	\]
	We write it as $C=\mathrm{face}(n_1,\cdots ,n_k)$.
\end{definition}

\begin{definition}
	For any non-empty subset $\boldsymbol{C}\subset \mathbb{R}^3 $, we define the \textit{polar cone} of $\boldsymbol{C}$ as,
	\[
		\boldsymbol{C}^*:=\left\{ x \in \mathbb{R}^3 : x\cdot y \le 0, \forall y \in \boldsymbol{C} \right\}.
	\]
\end{definition}

A simple geometric argument shows the following relation between $\boldsymbol{C}$ and $\boldsymbol{C}^*$.
\begin{proposition}
	Given a polyhedral cone $\boldsymbol{C}$ defined as
	\[
		\boldsymbol{C}=\mathrm{span}(u_1,\cdots ,u_k)=\mathrm{face}(n_1,\cdots ,n_l),
	\]
	we have
	\[
		\boldsymbol{C}^*=\mathrm{span}(n_1,\cdots ,n_l)=
		\mathrm{face}(u_1,\cdots ,u_k).
	\]
	
\end{proposition}


We say a polyhedral cone $\boldsymbol{C}$ is \textit{non-degenerate} if $\boldsymbol{C}$ and $\boldsymbol{C}^*$ have non-empty interiors.

\begin{remark}
	Note that since $\boldsymbol{C}^{**}=\boldsymbol{C}$, we know $\boldsymbol{C}$ is non-degenerate if and only if $\boldsymbol{C}^{*}$ is non-degenerate.
\end{remark}

In general, some of $n_j$ or $u_j$ are redundant in the definition of $\boldsymbol{C}$.
Indeed, we can choose them based on geometric observation.
For any non-degenerate polyhedral cone $\boldsymbol{C}$, we use $F_1,\cdots ,F_k$ to denote its faces and write $n_j$ as the unit normal of each face $F_j$, pointing outside of $\boldsymbol{C}$.
We can arrange $F_j$ such that $\mathrm{det}(n_j,n_{j+1},n_{j+2})<0$ for each $1\le j\le k$.
Here we use the subscript $(\cdot)_{k+j}=(\cdot)_{j}$.
Such choice of $n_j$ makes sure $n_j \times n_{j+1}$ pointing the same direction with $F_j \cap F_{j+1}$.

Now we can choose $u_j=\frac{n_j \times n_{j+1}}{\left|n_j \times n_{j+1}\right|}$, and we know $F_j \cap F_{j+1}=\left\{ r u_j:r\ge 0 \right\}$.
In this case, we call this cone a \textit{$k$-faced polyhedral cone} and we always use $F_j,u_j,n_j$ defined above from now on.
We use $\theta_j=\theta_j(\boldsymbol{C})=\measuredangle F_jF_{j+1}$ to denote the dihedral angles between its faces.

\begin{remark}
	An easy computation implies $\mathrm{det}(u_j,u_{j+1},u_{j+2})>0$. This  agrees with the conventional notations.
	For example, for the cone $\boldsymbol{C}=\{ x=(x_1,x_2,x_3) \in \mathbb{R}^3 : x_j\ge 0, \forall i=1,2,3 \}$, we usually choose $F_j=\left\{ x_j=0 \right\}, n_j=-e_j, u_j=e_{j+2}$.
	Note that although $u_j\times u_{j+1}$ parallels to $n_j$, they point to different directions.
\end{remark}

Since $k$-faced polyhedral cones are closely related to the spherical convex $k$-gons on $\mathbb{S}^2$, we give some basic results about spherical polygons.
We write $\boldsymbol{D}=\boldsymbol{D}_{\boldsymbol{C}}:=\boldsymbol{C}\cap \mathbb{S}^2$ for arbitrary cone $\boldsymbol{C}$.
Similarly, we write $\boldsymbol{D}^\dag:=\boldsymbol{C}^*\cap \mathbb{S}^2$ as the dual of the spherical $k$-gon $\boldsymbol{D}$.
\begin{proposition}
	$\boldsymbol{C}$ is a $k$-faced polyhedral cone if and only if $\boldsymbol{D}_{\boldsymbol{C}}$ is a non-degenerate spherical convex $k$-sided polygon on the unit sphere $\mathbb{S}^2$.
\end{proposition}

Note that the vertices of $\boldsymbol{D}$ are $u_1,\cdots, u_k$. We use $\widefrown{p,q}=\widefrown{pq}$ to denote the great-circle arc between $p,q$ on $\mathbb{S}^2$. 
Then it is easy to see the angle $\theta_j$ is just the angle between $\widefrown{u_j u_{j+1}}$ and $\widefrown{u_j u_{j-1}}$.
On the other hand, we also know $\theta_j=\pi-\measuredangle (n_j,n_{j+1})$ and hence $\pi-\theta_j=\left|\widefrown{n_j n_{j+1}}\right|$, the length of arc $\widefrown{n_j n_{j+1}}$.

\begin{definition}
	Suppose $\boldsymbol{C},\boldsymbol{C}'$ are two $k$-faced polyhedral cones for some $k \in \mathbb{N}$.
	We say $\boldsymbol{C}'$ has \textit{less dihedral angle} than $\boldsymbol{C}$ if $\theta_j\le \theta'_j$ for all $1\le j\le k$, where $\theta_j'$ is the dihedral angle between their side faces.

	We say $\boldsymbol{C}'$ has \textit{strictly less dihedral angle} than $\boldsymbol{C}$ if in addition, there is at least one of $j$ such that $\theta_j <\theta_j'$.
	
\end{definition}
\begin{remark}
	$\boldsymbol{C}'$ has less dihedral angle than $\boldsymbol{C}$ if and only if the spherical $k$-gon $\boldsymbol{D}'^\dag$ has longer side lengths than the corresponding sides of $\boldsymbol{D}^\dag$.
	\label{rmk_less_dihe}
\end{remark}

\subsection{Properties of pyramids}%
\label{sub:properties_of_pyramids}

In this subsection, we fix a $k$-faced polyhedral cone $\boldsymbol{C}$.
We use the concept of pyramids instead of cone type polyhedra to shorten our statement.

For any $\xi \in \mathbb{S}^2 $, let $\pi:=\left\{ x \in \mathbb{R}^3 :x\cdot \xi=-1 \right\}$ be the plane associated with $\xi$. It has a distance of 1 from $O$.

We write $P=P_{\xi,\boldsymbol{C}}:=\left\{ x \in \boldsymbol{C}:x\cdot \xi\ge -1\right\}$.
Then, we can find that $P$ is the region bounded by $\pi$ and $\boldsymbol{C}$.
In general, $P$ might not be a bounded polyhedron.
If $P$ is indeed a bounded polyhedron, then it is a $k$-gonal pyramid.

We have the following result by a simple geometric argument.
\begin{lemma}
	$P$ defined above is a pyramid if and only if $\xi$ is in the interior of $\boldsymbol{C}^*$.
\end{lemma}

\begin{proof}
	We write $A_{j}$ as the intersection of $F_j \cap F_{j+1}\cap \pi$ if $F_j \cap F_{j+1}\cap \pi\neq \emptyset $.
	Then we know $P$ is a pyramid if and only if $A_{j}$ exists for each $1\le j\le k$.
	Note that $A_{j}$ can be computed by $A_{j}=-\frac{u_j}{u_j\cdot \xi}$.
	Hence, $P$ is a pyramid if and only if $u_j \cdot \xi<0$ for each $1\le j\le k$.

	Note that the condition $u_j\cdot\xi =0$ is equivalent to the fact that $\xi$ is lying in the plane spanned by $n_j,n_{j+1}$, the plane in which one of the faces of $\boldsymbol{C}^*$ locates in.
	Hence $u_j\cdot \xi<0, \forall 1\le j\le k$ is equivalent to the condition that $\xi \in \boldsymbol{C}^*$ and $\xi$ cannot on the faces of $\boldsymbol{C}^*$.
\end{proof}

If $P$ is a pyramid, we write $B=\pi \cap \boldsymbol{C}$ as the base face of $P$. In addition, $\xi$ is the inner unit normal vector of the base face $B$.
We define $\gamma_j:=\measuredangle F_j B=
\measuredangle (n_j,\xi)$.
In particular, we know $\gamma_j=\left|\widefrown{\xi n_j}\right|$.

Now we fix some underlying $k$-faced polyhedral cone $\bar{\boldsymbol{C}}$ and $k$-gonal pyramid $\bar{P}$.
We would like to define energy $E$ for any pyramid $P$ with respect to $\bar{P}$ as
\begin{equation}
	E(P)=E_{\bar{P}}(P):=\left|B\right|-\sum_{j =1}^{k}\cos \bar{\gamma}_j \left|F_j \cap P\right|.\label{functional in flat background}
\end{equation}

 When we try to reduce the dihedral angles of $\boldsymbol{C}$, we would expect the energy should decrease.
Specifically, we make the following conjecture.
\begin{conjecture}\label{minimiser existence conjecture}
	Suppose $\boldsymbol{C}$ is a $k$-faced polyhedral cone such that it has less dihedral angle than $\bar{\boldsymbol{C}}$.
	Then we have
	\[
		\inf _{\xi \in \mathring {\boldsymbol{C}^*}}
		E_{\bar{P}}(P_{\xi,\boldsymbol{C}})\le 0.
	\]
	The equality holds if and only if $\bar{\boldsymbol{C}}$ is isometric to $\boldsymbol{C}$.
 Here, {\itshape{isometric}} means that we can find an orthogonal matrix $\Omega$ such that $\Omega(\bar{\boldsymbol{C}})=\boldsymbol{C}$.
	\label{conj_rig_pyramid}
\end{conjecture}
We do have some information for $E_{\bar{P}}(P_{\xi,\boldsymbol{C}})$ near $\bar{P}$.
Let $P_t = P_{\nu_t, \boldsymbol{C}_t}$ be a differentiable family of polyhedra such that
$P_0 = \bar{P}$, $E (t) = E_{\bar{P}} (P_t)$ and $\theta_j (t) = \theta_j
(\boldsymbol{C}_t)$. Applying the Schlafli formula \cite{rivin-schlafli-2000} for polyhedra in the Euclidean 3-space,
as we decrease $\theta_j (t)$ near $t=0$, the energy $E (t)$ decreases. We compute the energy $E$ first.

\begin{lemma}
	\label{lem_energy_compute}
	$E$ can be computed by
	\begin{equation}
		E=E_{\xi,\bar{\boldsymbol{C}}}=\sum_{j =1}^{k}
		\left|F_j \cap P\right|(\cos \gamma_j-\cos \bar{\gamma}_j).
		\label{eq_lem_energy}
	\end{equation}
\end{lemma}
\begin{proof}
	Let $O'$ be the projection of $O$ to the plane $\pi$. Then we find
	\[
	\left|B\right|=\sum_{j =1}^{k}
		\text{signed area}(\triangle A_{j}A_{j+1}O')
		=\sum_{j =1}^{k}\cos \gamma_j \left|F_j \cap P\right|.
	\]
	
	Note that we allow $O'$ to be outside of base face $B$. 
	In this case, we need to consider the signed area here for triangle $A_jA_{j+1}O'$.

	Together with the definition of $E$, we can get formula (\ref{eq_lem_energy}).
\end{proof}

\begin{proposition}
  Let $\ell_j (t)$ be length of the side edeges of $P_t$, then
  \begin{equation*}
    E' (0) = \tfrac{1}{2} \sum_j \ell_j (0) \theta_j' (0) .
  \end{equation*}
\end{proposition}

\begin{proof}
First from
\begin{equation*}
  E (t) = \sum_j |F_j |  (\cos \gamma_j - \cos \bar{\gamma}_j),
\end{equation*}
we have
\begin{equation*}
  \tfrac{\mathrm{d}}{\mathrm{d} t} E (t) = \sum_j \left( (\cos \gamma_j - \cos
  \bar{\gamma}_j) \tfrac{\mathrm{d}}{\mathrm{d} t} |F_j | - |F_j | \sin
  \gamma_j \tfrac{\mathrm{d}}{\mathrm{d} t} \gamma_j \right) .
\end{equation*}
Let $h_j$ be the distance from $O$ to the edge $B \cap F_j$ and $l_j$ be the
length of the edge $F_j \cap B$, from computing the area of $F_j$ we see
\begin{equation*}
  |F_j | \sin \gamma_j = \tfrac{1}{2} l_j h_j \sin \gamma_j .
\end{equation*}
Since $O$ is of distance 1 to $B$, so $h_j \sin \gamma_j = 1$. So $|F_j | \sin
\gamma_j = \tfrac{1}{2} l_j$ and
\begin{equation*}
  \tfrac{\mathrm{d}}{\mathrm{d} t} E (t) = \sum_j \left( (\cos \gamma_j - \cos
  \bar{\gamma}_j) \tfrac{\mathrm{d}}{\mathrm{d} t} |F_j | - \tfrac{1}{2} l_j
  \tfrac{\mathrm{d}}{\mathrm{d} t} \gamma_j \right) .
\end{equation*}
Combining the Schlafli formula which says that
\begin{equation*}
  \sum_j \ell_j \tfrac{\mathrm{d}}{\mathrm{d} t} \theta_j + \sum_j l_j \tfrac{\mathrm{d}}{\mathrm{d} t} 
  \gamma_j = 0
\end{equation*}
and $\gamma_j (0) = \bar{\gamma}_j$ finishes the
proof.
\end{proof}

If we assume $k=3$, we have a positive answer to Conjecture \ref{conj_rig_pyramid}.
\begin{proposition}
	We fix an underlying $3$-faced polyhedral cone and a pyramid $\bar{P}$.
	Suppose $\boldsymbol{C}$ is another $3$-faced polyhedral cone.
	If $\boldsymbol{C}$ has less dihedral angle than $\bar{\boldsymbol{C}}$, then we can find a unit vector $\xi$ in the interior of $\boldsymbol{C}^*$ such that
	\begin{equation}
		E_{\bar{P}}(P_{\xi,\boldsymbol{C}})\le 0.
		\label{eq_energy_est}
	\end{equation}
	If $\boldsymbol{C}$ is not isometric to $\bar{\boldsymbol{C}}$, then the inequality (\ref{eq_energy_est}) can be strict.
	
	\label{prop_energy_est}
\end{proposition}

Since the proof of Proposition \ref{prop_energy_est} is quite long, we put it into Appendix \ref{proof of minimiser in general tetrahedron}.
Note that, we have assumed an additional condition \eqref{extra dihedral angle condition} on dihedral angles.
In this case, 
we have a simpler geometric proof.

\begin{proposition}
	\label{prop_pyramid_another_condition}
	Suppose $\boldsymbol{C}$ has less dihedral angle than $\bar{\boldsymbol{C}}$, and ${\boldsymbol{C}}$ satisfies
	\[
		\left|\pi-(\bar{\gamma}_j+\bar{\gamma}_{j+1})\right|\le \theta_j.
	\]
	Then we can find a unit vector $\xi$ in the interior of $\boldsymbol{C}^*$ such that
	\[
		E_{\bar{P}}(P_{\xi,\boldsymbol{C}})\le 0.
	\]
	If $\boldsymbol{C}$ is not isometric to $\bar{\boldsymbol{C}}$, then this inequality can be strict.
\end{proposition}

In view of Lemma \ref{lem_energy_compute}, we only need to find $\xi$ such that $\gamma_j>\bar{\gamma}_j$ if $\boldsymbol{C}$ is not isometric to $\bar{\boldsymbol{C}}$.
Hence, the proof for Proposition \ref{prop_pyramid_another_condition} is a direct consequence of the following lemma by mathematical induction.

\begin{lemma}
	\label{lem_pyramid_one_side}
	Suppose $\boldsymbol{C},\bar{\boldsymbol{C}}$ are two 3-faced polyhedral cones such that $\theta_1<\bar{\theta}_1$, $\theta_2=\bar{\theta}_2$, $\theta_3=\bar{\theta}_3$.
	Let $\bar{P}=\bar{P}_{\bar{\xi},\bar{\boldsymbol{C}}}$ be a pyramid associated with $\boldsymbol{C}$.
	Suppose $\theta_1$ satisfies $\left|\pi-(\bar{\gamma}_1+\bar{\gamma}_2)\right|<\theta_1$.
	Then, we can find a vector $\xi \in \mathring {\boldsymbol{C}^*}$ and a pyramid $P=P_{\xi,\boldsymbol{C}}$ such that
	\[
		\gamma_1>\bar{\gamma}_1, \gamma_2>\bar{\gamma}_2, \gamma_3>\bar{\gamma}_3.
	\]
\end{lemma}

\begin{proof}
	Note that the condition $\left|\pi-(\bar{\gamma}_1+\bar{\gamma}_2)\right|<\theta_1$ is equivalent to
	\[
		\left|\widefrown{n_1n_2}\right|<\left|\widefrown{\bar{\xi}\bar{n}_1}\right|+\left|\widefrown{\bar{\xi}\bar{n}_2}\right|.
	\]
	
	This lemma can be formulated using the properties of dual cones as the following lemma.
\end{proof}
	\begin{lemma}
	\label{lem_pyramid_another_stat}
	We fix two spherical convex triangles $\boldsymbol{D}^\dag=n_1n_2n_3$ and $\bar{\boldsymbol{D}}^\dag=\bar{n}_1\bar{n}_2\bar{n}_3$ with
	\[
		\left|\widefrown{n_1n_3}\right|=\left|\widefrown{\bar{n}_1\bar{n}_3}\right|,
		\left|\widefrown{n_2n_3}\right|=\left|\widefrown{\bar{n}_2\bar{n}_3}\right|,
		\left|\widefrown{n_1n_2}\right|>\left|\widefrown{\bar{n}_1\bar{n}_2}\right|.
	\]
	Given $\bar{\xi} \in \mathring{\bar{\boldsymbol{D}}}^\dag$ such that
	\begin{equation}
		\left|\widefrown{n_1n_2}\right|<\left|\widefrown{\bar{\xi}\bar{n}_1}\right|+\left|\widefrown{\bar{\xi}\bar{n}_2}\right|.
		\label{eq_lem_angle_cond}
	\end{equation}
	Then we can find $\xi \in \mathring {\boldsymbol{D}}^\dag$ such that
	\[
		\left|\widefrown{\xi n_1}\right|>
		\left|\widefrown{\bar{\xi}\bar{n}_1}\right|, \left|\widefrown{\xi n_2}\right|>
		\left|\widefrown{\bar{\xi}\bar{n}_2}\right|, \left|\widefrown{\xi n_3}\right|>
		\left|\widefrown{\bar{\xi}\bar{n}_3}\right|, 
	\]
	\end{lemma}
	\begin{proof}
		[Proof of Lemma \ref{lem_pyramid_another_stat}]

		Note that since $\left|\widefrown{\bar{n}_1\bar{n}_3}\right|=\left|\widefrown{n_1n_3}\right|$, we can find a unit vector $\xi \in \mathbb{S}^2$ such that the triangle $\xi n_1n_3$ is congruent to triangle $\bar{\xi} \bar{n}_1 \bar{n}_3$.
		Moreover, $\xi$ is unique if we require $\mathrm{det}(\xi,n_1,n_3)$ has the same sign with $\mathrm{det}(n_2,n_1,n_3)$.
		This makes sure $\xi$ and $n_2$ lie the same side of the greater circle passing through $n_1,n_3$.

		Now, we show $\left|\widefrown{\xi n_3}\right|>\left|\widefrown{\bar{\xi}\bar{n}_3}\right|$ and $\xi \in \mathring {\boldsymbol{D}}^\dag$.
		We draw a 2-D example to show the idea of such proof (see Figure \ref{fig:2_dtriangle}).
\begin{figure}[h!]
    \centering
	\begingroup
	\fontsize{9pt}{12pt}
	\def\svgwidth{0.8\columnwidth}
\begingroup%
  \makeatletter%
  \providecommand\color[2][]{%
    \errmessage{(Inkscape) Color is used for the text in Inkscape, but the package 'color.sty' is not loaded}%
    \renewcommand\color[2][]{}%
  }%
  \providecommand\transparent[1]{%
    \errmessage{(Inkscape) Transparency is used (non-zero) for the text in Inkscape, but the package 'transparent.sty' is not loaded}%
    \renewcommand\transparent[1]{}%
  }%
  \providecommand\rotatebox[2]{#2}%
  \newcommand*\fsize{\dimexpr\f@size pt\relax}%
  \newcommand*\lineheight[1]{\fontsize{\fsize}{#1\fsize}\selectfont}%
  \ifx\svgwidth\undefined%
    \setlength{\unitlength}{680.31496063bp}%
    \ifx\svgscale\undefined%
      \relax%
    \else%
      \setlength{\unitlength}{\unitlength * \real{\svgscale}}%
    \fi%
  \else%
    \setlength{\unitlength}{\svgwidth}%
  \fi%
  \global\let\svgwidth\undefined%
  \global\let\svgscale\undefined%
  \makeatother%
  \begin{picture}(1,0.33333333)%
    \lineheight{1}%
    \setlength\tabcolsep{0pt}%
    \put(0,0){\includegraphics[width=\unitlength,page=1]{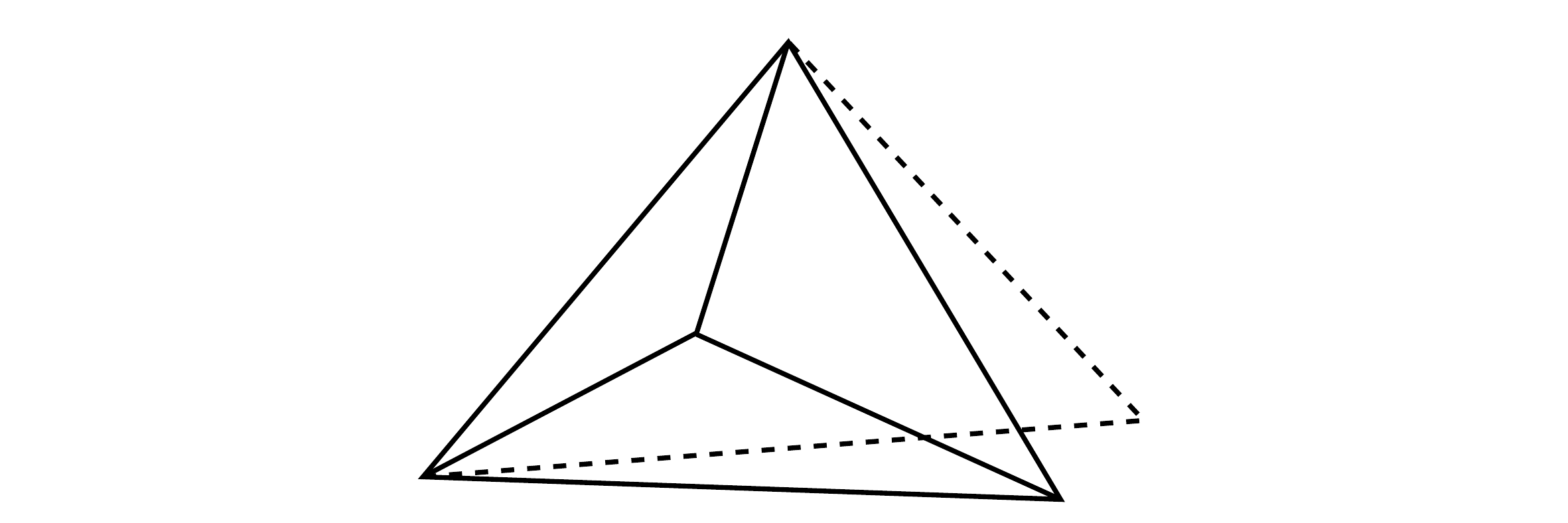}}%
    \put(0.48086338,0.3187627){\color[rgb]{0,0,0}\makebox(0,0)[lt]{\lineheight{0}\smash{\begin{tabular}[t]{l}$\bar{n}_3(n_3)$\end{tabular}}}}%
    \put(0.16981861,0.02235058){\color[rgb]{0,0,0}\makebox(0,0)[lt]{\lineheight{0}\smash{\begin{tabular}[t]{l}$\bar{n}_1(n_1)$\end{tabular}}}}%
    \put(0.4617413,0.12371754){\color[rgb]{0,0,0}\makebox(0,0)[lt]{\lineheight{0}\smash{\begin{tabular}[t]{l}$\bar{\xi}(\xi)$\end{tabular}}}}%
    \put(0.68999029,0.00832901){\color[rgb]{0,0,0}\makebox(0,0)[lt]{\lineheight{0}\smash{\begin{tabular}[t]{l}$\bar{n}_2$\end{tabular}}}}%
    \put(0.73556819,0.06671909){\color[rgb]{0,0,0}\makebox(0,0)[lt]{\lineheight{0}\smash{\begin{tabular}[t]{l}$n_2$\end{tabular}}}}%
    \put(0,0){\includegraphics[width=\unitlength,page=2]{2_dtriangle.pdf}}%
  \end{picture}%
\endgroup%

	\endgroup

    \caption{2-dimensional examples}
    \label{fig:2_dtriangle}
\end{figure}

Recall the spherical law of cosines says
\[
	\cos \left|\widefrown{n_1n_2}\right|=
	\cos \left|\widefrown{n_1n_3}\right|
	\cos \left|\widefrown{n_2n_3}\right|+
	\sin \left|\widefrown{n_1n_3}\right|
	\sin \left|\widefrown{n_2n_3}\right|
	\cos \measuredangle (n_1n_3n_2).
\]

Hence, if we increase the length of $\widefrown{n_1n_2}$, we know the angle $\measuredangle (n_1n_3n_2)$ is increasing, too.
So we know the angle $\measuredangle (\xi n_3n_2)$ is increasing. Again by spherical law of cosines, we know $\left|\widefrown{\xi n_2}\right|$ should increase.
This shows $\left|\widefrown{\xi n_2}\right|>\left|\widefrown{\bar{\xi}\bar{n}_2}\right|$.

To show $\xi \in \mathring {\boldsymbol{D}}^\dag$, we write the length of $\widefrown{n_1n_2}$ as $\left|\widefrown{\bar{n}_1\bar{n}_2}\right| +t$ for some $t\ge 0$.
Clearly, if $t=0$, we know $\xi$ is in the interior of $\boldsymbol{D}^\dag$.
Let $t_0$ be the smallest $t$ such that $\xi$ is on the boundary of $\boldsymbol{D}^\dag$.
For such $t$, we know the only possible way is that $\xi$ locates on the arc $\widefrown{n_1n_2}$.
This implies
\[
	\left|\widefrown{n_1n_2}\right|=
	\left|\widefrown{n_1 \xi}\right|+
	\left|\widefrown{\xi n_2}\right|>
	\left|\widefrown{\bar{n}_1 \bar{\xi}}\right|+
	\left|\widefrown{\bar{\xi}\bar{n}_2}\right|.
\]

This inequality contradicts the condition (\ref{eq_lem_angle_cond}).
Hence $\xi$ cannot touch the boundary of $\boldsymbol{D}^\dag$.
Now, since $\xi$ is in the interior of $\boldsymbol{D}^\dag$, we can adjust $\xi$ a bit to make sure
	\[
		\left|\widefrown{\xi n_1}\right|>
		\left|\widefrown{\bar{\xi}\bar{n}_1}\right|, \left|\widefrown{\xi n_2}\right|>
		\left|\widefrown{\bar{\xi}\bar{n}_2}\right|, \left|\widefrown{\xi n_3}\right|>
		\left|\widefrown{\bar{\xi}\bar{n}_3}\right|.
	\]
\end{proof}

\begin{proof}[Proof of Proposition \ref{prop_pyramid_another_condition}]
	If $\boldsymbol{C}$ is not isometric to $\bar{\boldsymbol{C}}$, we know there exists at least one $j$ such that $\theta_j<\bar{\theta}_j$.

	Define $\boldsymbol{C}^k$ to be the 3-faced polyhedral cones with $\theta^k_j=\theta_j$ if $j\le k$ and $\theta^k_j=\bar{\theta}_j$ where $\theta^k_j$ is the dihedral angles of $\boldsymbol{C}^k$.

	Then, we can find $\xi^k$ such that
	\[
		\left|\widefrown{\xi^k n_l}\right|>
		\left|\widefrown{\xi^{k-1}n_l}\right|
	\]
	for any $1\le l,k\le 3$ based on Lemma \ref{lem_pyramid_one_side}. Here we have chosen $\xi^0=\bar{\xi}$.

	Hence, $\xi=\xi^3$ is the vector we want.
\end{proof}

\subsection{An example}%
\label{sub:some_examples}

Now we give an example to explain
why $k>3$ is difficult to Conjecture \ref{conj_rig_pyramid}.
We consider the example with $k=4$ here.

\begin{example}\label{tetrahedron example}
Given two constant $\beta_1,\beta_2 \in (0,\frac{\pi}{2})$, we define
$\boldsymbol{C}=\boldsymbol{C}^{\beta_1,\beta_2}$ by choosing
\begin{align*}
	n_1={}&(\sin \beta_1,0, -\cos \beta_1),\\
	n_2={}& (0, \sin \beta_2, -\cos \beta_2) ,\\
	n_3={}& (-\sin \beta_1,0,-\cos \beta_1), \\
	n_4={}& (0,-\sin \beta_2,-\cos \beta_2).
\end{align*}
\end{example}

\begin{proposition}
  If $\cos \beta_1\cos \beta_2=\cos \beta_1'\cos \beta_2'$ for some $\beta_1,\beta_2,\beta_1',\beta_2' \in (0,\frac{\pi }{2})$, then the polyhedral cones $\boldsymbol{C}^{\beta_1,\beta_2}$ and $\boldsymbol{C}^{\beta_1',\beta_2'}$ have the same dihedral angles but are not isometric to each other if $\beta_1\neq \beta_1'$ and $\beta_1\neq \beta_2'$.
  \label{prop_non_isometric_cone}
\end{proposition}

\begin{proof}
    Note $n_i\cdot n_{i+1}=\cos \beta_1\cos \beta_2=\cos \beta_1'\cos \beta_2'=n'_i\cdot n_{i+1}'$ where $n_i'$ are the face vectors of $\boldsymbol{C}^{\beta_1',\beta_2'}$.
    Hence, they have the same dihedral angles.
    
    Geometrically speaking, $\boldsymbol{D}^\dag$ and $\boldsymbol{D}'^\dag$ are spherical rhombuses and they have the same side length but have different interior angles.
    Clearly, they are not congruent with each other.
\end{proof}

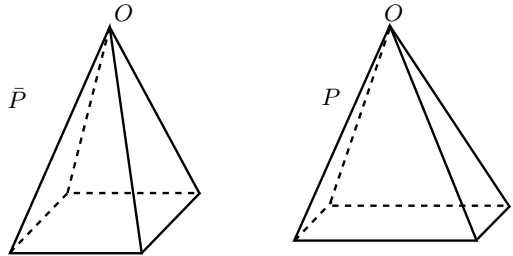
\begin{figure}[h!]
    \centering
	\begingroup
	\fontsize{9pt}{12pt}
	\def\svgwidth{0.8\columnwidth}
\begingroup%
  \makeatletter%
  \providecommand\color[2][]{%
    \errmessage{(Inkscape) Color is used for the text in Inkscape, but the package 'color.sty' is not loaded}%
    \renewcommand\color[2][]{}%
  }%
  \providecommand\transparent[1]{%
    \errmessage{(Inkscape) Transparency is used (non-zero) for the text in Inkscape, but the package 'transparent.sty' is not loaded}%
    \renewcommand\transparent[1]{}%
  }%
  \providecommand\rotatebox[2]{#2}%
  \newcommand*\fsize{\dimexpr\f@size pt\relax}%
  \newcommand*\lineheight[1]{\fontsize{\fsize}{#1\fsize}\selectfont}%
  \ifx\svgwidth\undefined%
    \setlength{\unitlength}{680.31496063bp}%
    \ifx\svgscale\undefined%
      \relax%
    \else%
      \setlength{\unitlength}{\unitlength * \real{\svgscale}}%
    \fi%
  \else%
    \setlength{\unitlength}{\svgwidth}%
  \fi%
  \global\let\svgwidth\undefined%
  \global\let\svgscale\undefined%
  \makeatother%
  \begin{picture}(1,0.33333333)%
    \lineheight{1}%
    \setlength\tabcolsep{0pt}%
    \put(0,0){\includegraphics[width=\unitlength,page=1]{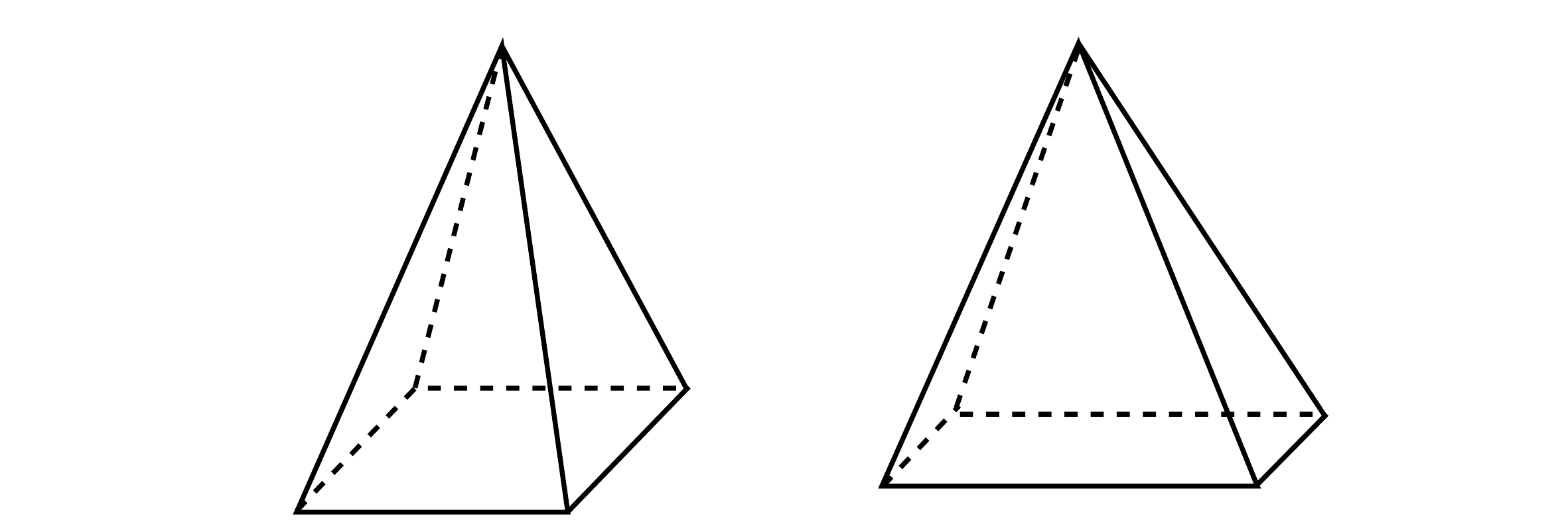}}%
    \put(0.18619524,0.19726707){\color[rgb]{0,0,0}\makebox(0,0)[lt]{\lineheight{0}\smash{\begin{tabular}[t]{l}$\bar{P}$\end{tabular}}}}%
    \put(0.59837115,0.20202102){\color[rgb]{0,0,0}\makebox(0,0)[lt]{\lineheight{0}\smash{\begin{tabular}[t]{l}$P$\end{tabular}}}}%
    \put(0.32596954,0.31241796){\color[rgb]{0,0,0}\makebox(0,0)[lt]{\lineheight{0}\smash{\begin{tabular}[t]{l}$O$\end{tabular}}}}%
    \put(0.68079374,0.31241796){\color[rgb]{0,0,0}\makebox(0,0)[lt]{\lineheight{0}\smash{\begin{tabular}[t]{l}$O$\end{tabular}}}}%
  \end{picture}%
\endgroup%

	\endgroup

    \caption{Two pyramids with the same dihedral angles between their side faces}
    \label{fig:two-pyramids}
\end{figure}

Proposition \ref{prop_non_isometric_cone} shows, when we know two polyhedral cones have the same dihedral angles between their adjacent faces, we cannot conclude these two polyhedral cones are isometric. See Figure \ref{fig:two-pyramids} for this example.


Another hard part is that the results similar to Lemma \ref{lem_pyramid_one_side} and Lemma \ref{lem_pyramid_another_stat} no longer hold.

\begin{proposition}
  There exist two $4$-faced polyhedral cones $\boldsymbol{C}, \bar{\boldsymbol{C}}$ and a pyramid $\bar{P}=\bar{P}_{\bar{\xi},\bar{\boldsymbol{C}}}$ with following conditions.
  \begin{itemize}
      \item $\theta_j<\bar{\theta}_j$, for $1\le j\le 4$ ($\boldsymbol{C}$ has less dihedral angle than $\bar{\boldsymbol{C}}$).
      \item $|\pi-(\bar{\gamma}_j+\bar{\gamma}_{j+1})|<\theta_j$, for $1\le j\le 4$.
  \end{itemize}
  Then we cannot find a unit vector $\xi\in \mathring{\boldsymbol{C}^*}$ such that the pyramid $P_{\xi,\boldsymbol{C}}$ satisfies
  $$
  \gamma_j\ge \bar{\gamma}_j,\quad \text{ for all }1\le j\le 4.
  $$
  \label{prop_quad_1}
\end{proposition}

\begin{proposition}
  There exist two spherical convex quadrilaterals $\boldsymbol{D}^\dag=n_1n_2n_3n_4$ and $\bar{\boldsymbol{D}}^\dag=\bar{n}_1\bar{n}_2\bar{n}_3\bar{n}_4$, and a unit vector $\bar{\xi} \in \mathring{\bar{\boldsymbol{D}}}^\dag$ with the following conditions.
  \begin{itemize}
      \item $\left|\widefrown{n_jn_{j+1}}\right|\ge \left|\widefrown{\bar{n}_j\bar{n}_{j+1}}\right|$, for $1\le j\le 4$.
      \item
      $\left|\widefrown{n_jn_{j+1}}\right|<\left|\widefrown{\bar{\xi}\bar{n}_j}\right|+\left|\widefrown{\bar{\xi}\bar{n}_{j+1}}\right|$, for $1\le j\le 4$.
  \end{itemize}
  Then, we cannot find $\xi \in \mathring{\boldsymbol{D}}^\dag$ such that
  \begin{equation}
      \left|\widefrown{\xi n_j}\right|\ge \left|\widefrown{\bar{\xi}\bar{n}_j}\right|, \quad
  \text{ for all }1\le j\le 4.
  \label{eq_prop_length}
  \end{equation}
  \label{prop_quad_2}
\end{proposition}
\begin{proof}
[Proof of Proposition \ref{prop_quad_1} and Proposition \ref{prop_quad_2}]

Since Proposition \ref{prop_quad_1} and Proposition \ref{prop_quad_2} are equivalent, we only need to show Proposition \ref{prop_quad_2}.

Note that if we choose $n_i$ such that $\left|\widefrown{n_jn_{j+1}}\right|\ge \left|\widefrown{\bar{n}_j\bar{n}_{j+1}}\right|$, for $1\le j\le 4$, then automatically the second condition holds.
Hence, we can use $\boldsymbol{C}^{\beta_1,\beta_2}$ to construct such examples.
In other words, we always choose $\beta_j$ and $\bar{\beta}_j$ such that
$\cos \beta_1\cos \beta_2=\cos \bar{\beta}_1\cos \bar{\beta}_2$.

Now, we consider a Euclidean geometry problem.
Fix a square $\bar{S}=\bar{A}_1\bar{A}_2\bar{A}_3\bar{A}_4
\subset \mathbb{R}^2$ with center $O$, we consider a rhombus $S=A_1A_2A_3A_4$ having the same side length with $\bar{S}$.
If we choose the rhombus such that one of the diagonals is much shorter than its side length, then clearly, we cannot find a point $P$ in the interior of $S$ such that
$$|PA_j|\ge |O\bar{A}_j|, \quad 
\text{ for all }1\le j\le 4.$$
See Figure \ref{fig:two-quad} for the shape of $\bar{S}$ and $S$.
\begin{figure}[h!]
    \centering
	\begingroup
	\fontsize{9pt}{12pt}
	\def\svgwidth{0.8\columnwidth}
	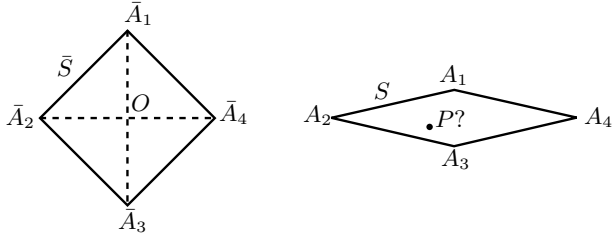
	\endgroup

    \caption{A square $\bar{S}$ and a rhombus $S$}
    \label{fig:two-quad}
\end{figure}

Now we can choose  $\beta_1>\beta_2$ and $\bar{\beta}_1=\bar{\beta}_2$ sufficient small and $\frac{\beta_1}{\beta_2}$ sufficient large such that the shape of $\bar{\boldsymbol{D}}^\dag$ and $\boldsymbol{D}^\dag$ are sufficiently close to $\bar{S}$ and $S$, respectively.

For such choice of $n_i$ and $\bar{n}_i$, we cannot find $\xi \in \mathring{\boldsymbol{D}}^\dag$ satisfying \eqref{eq_prop_length}.
\end{proof}

\section{Dihedral non-rigidity}\label{sec:dihedral non-rigid}

In this section, we give proof of Theorem \ref{dihedral non-rigidity}. The proof dates back to the work of Schoen-Yau of the positive mass theorem \cite{schoen-proof-1979}.

\begin{proof}[Proof of Theorem \ref{dihedral non-rigidity}]
We argue by contradiction and assumes that $\measuredangle_{ij}M < \measuredangle_{ij}\bar{P}$. Using Theorem \ref{thm_regularity}, the minimiser to the problem \eqref{action} is achieved by $E$ such that $E$ is an open set and $\Sigma=\partial E \cap \mathring{M}$ is $C^{1,\alpha}$ to the corner. So $\Sigma$ is a stable capillary surface with constant mean curvature
  $- 2$. By the second variation formula, we have
  \begin{equation}
    \int_{\Sigma} [| \nabla f|^2 - (|A|^2 +\ensuremath{\operatorname{Ric}}(N))
    f^2] - \int_{\partial \Sigma} Q f^2 \geq 0 \label{stability
    inequality}
  \end{equation}
  for any function compactly supported away from the corners, where on
  $\partial \Sigma \cap F_j$,
  \begin{equation*}
    Q = \tfrac{1}{\sin \bar{\gamma }_j} {A_{\partial M}} (\tilde{\nu}, \tilde{\nu}) - A (\nu, \nu) \cot
    \bar{\gamma }_j .
  \end{equation*}
  Since $\Sigma$ is $C^{1, \alpha}$ to its corners, its curvature is $|A|$ is
  square integrable. By approximation, we can take the function $f$ in the
  stability inequality to be $f \equiv 1$ which gives
\begin{align*}
- & \int_{\Sigma} [(|A|^2 +\ensuremath{\operatorname{Ric}}(N))] -
\int_{\partial \Sigma} [\tfrac{1}{\sin \bar{\gamma }_j} {A_{\partial M}} (\tilde{\nu}, \tilde{\nu})
- A (\nu, \nu) \cot \bar{\gamma }_j] \geq 0.
\end{align*}
  Applying the Gauss equation (see {\cite{schoen-proof-1979}}), we have
  \begin{equation*}
    |A|^2 +\ensuremath{\operatorname{Ric}} (N) = \tfrac{1}{2} (R_g - 2 K +
    |A|^2 + H^2)
  \end{equation*}
  where $R_g$ is the scalar curvature of $M$ and $K_{\Sigma}$ is the Gauss
  curvature of $\Sigma$.
  
  Since $H = - 2$ and $A = A^0 + \tfrac{1}{2} H g$ where $A^0$ denote the
  trace free part of $A$, the above turns into
  \begin{equation}
    |A|^2 +\ensuremath{\operatorname{Ric}} (N) = \tfrac{1}{2} (R + 6) -
    K_{\Sigma} + \tfrac{1}{2} |A^0 |^2 .\label{acg rewrite}
  \end{equation}
  The above formula is not new, see for example
  {\cite{andersson-rigidity-2008}}. By the Gauss-Bonnet formula for $C^{1,
  \alpha}$ surfaces with piecewise smooth boundary components, we have that
  \begin{equation*}
    \int_{\Sigma} K + \int_{\partial \Sigma} \kappa + \sum_{j = 1}^k (\pi -
    \alpha_j) = 2 \pi\chi (\Sigma) \leq 2 \pi,
  \end{equation*}
  here $\kappa$ is the geodesic curvature of $\partial \Sigma$ in $\Sigma$ and
  $\alpha_j$ are the interior angles of $\Sigma$ at the corners. By \cite[Lemma 3.2]{li-polyhedron-2020},
  $\alpha_j < \bar{\alpha}_j$ where $\bar{\alpha}_j$ is the corresponding interior angle
  of the base face of the polyhedron $\bar{P}$. Since
  \[ \sum_{j = 1}^k (\pi - \alpha_j) > \sum_{j = 1}^k (\pi - \bar{\alpha}_j) = 2
     \pi, \]
  we have that
  \begin{equation*}
    - \int_{\Sigma} K > \int_{\partial \Sigma} \kappa .
  \end{equation*}
  We conclude that 
  \begin{equation}
    \tfrac{1}{2} \int_{\Sigma} (R_g + 6 + |A^0 |^2) + \sum_{j = 1}^k
    \int_{\partial \Sigma} [\tfrac{1}{\sin \bar{\gamma }_j} {A_{\partial M}} (\tilde{\nu}, \tilde{\nu})
    - A (\nu, \nu) \cot \bar{\gamma }_j + \kappa_g] < 0. \label{boundary not
    processed}
  \end{equation}
  Since $\Sigma$ is of mean curvature $- 2$, we have $A (\nu, \nu) = - 2 - A
  (T, T)$. So
\begin{align*}
& \cos \bar{\gamma } A (\nu, \nu) + \kappa \sin \bar{\gamma } \\
= & 2 \cos \bar{\gamma } - \cos \bar{\gamma } A (T, T) + \langle \nabla_T \nu, T \rangle
\sin \bar{\gamma } \\
= & 2 \cos \bar{\gamma } - \langle \nabla_T T, \cos \bar{\gamma } N + \sin \bar{\gamma } \nu
\rangle \\
= & 2 \cos \bar{\gamma } - \langle \nabla_T T, X \rangle \\
= & 2 \cos \bar{\gamma } + A_{\partial M}(T, T) .
\end{align*}
  Since $T$ and $\tilde{\nu}$ form an orthonormal basis of $\partial M$, we have
  that
  \begin{equation}
    A_{\partial M}(\tilde{\nu}, \tilde{\nu}) - \cos \bar{\gamma } A (\nu, \nu) + \kappa \sin \bar{\gamma } =
    2 \cos \bar{\gamma } + {H_{\partial M}} . \label{hyperoblic contribution of the boundary}
  \end{equation}
  From \eqref{boundary not processed},
  \begin{equation*}
    \tfrac{1}{2} \int_{\Sigma} (R_g + 6 + |A^0 |^2) + \sum_{j = 1}^k
    \int_{\partial \Sigma} \tfrac{1}{\sin \bar{\gamma }_j} (H_j + 2 \cos
    \bar{\gamma }_j) < 0,
  \end{equation*}
  where $H_j=H_{F_j}$,
  contradicting that the scalar curvature $R_g \geq - 6$ and the mean
  curvature $H_j \geq - 2 \cos \bar{\gamma }_j$ of face $F_j$.
\end{proof}

\section{Rigidity}\label{sec:rigidity}
In this section, we show the rigidity statement Theorem \ref{rigidity} via constructing a foliation near the minimiser of \eqref{action in terms of cacciopolli}. We have two cases to consider: the case that the minimiser is an infinitesimally rigid surface and the case that the minimiser is  the vertex.
\subsection{Infinitesimally rigid CMC capillary surfaces}

Assume that the minimiser of \eqref{action} exists. Tracing back the
equalities in the proof of Theorem \ref{dihedral non-rigidity}, we conclude
that
\begin{equation}
  \chi (\Sigma) = 0, \text{ } R_M = - 6, \text{ } |A^0 | = 0 \text{ on }
  \Sigma, \label{interior rigid}
\end{equation}
and
\begin{equation}
  H_j = - 2 \cos \bar{\gamma }_j \text{ on } \partial \Sigma, \text{ }
  \alpha_j = \bar{\alpha}_j \text{ at the corners of } \Sigma . \label{boundary
  corner rigid}
\end{equation}
Moreover by the second variation formula \eqref{second variation},
\begin{equation*}
  \mathcal{Q} (f, f) = - \int_{\Sigma} f \Delta f + (|A|^2
  +\ensuremath{\operatorname{Ric}}(N)) f^2 + \sum_{j = 1}^k \int_{\partial
  \Sigma \cap F_j} f (\tfrac{\partial f}{\partial \nu} - Q f),
\end{equation*}
with $\mathcal{Q} (1, 1) = 0$. We can conclude that for any $C^2$ function $f$ compactly
supported away from the vertices of $\Sigma$, $\mathcal{Q} (1, f) = 0$. By choosing
an appropriate $f$, we further conclude that
\begin{equation*}
  \ensuremath{\operatorname{Ric}} (N) = - 2 \text{ on } \Sigma, \text{ }
  \tfrac{1}{\sin \bar{\gamma }_j} {A_{\partial M}} (\tilde{\nu}, \tilde{\nu}) - \cot \bar{\gamma }_j A (\nu,
  \nu) = 0 \text{ on } \partial \Sigma \cap F_j .
\end{equation*}
Combining with \eqref{acg rewrite} and \eqref{hyperoblic contribution of the boundary}, we conclude that
\begin{equation}
  K_{\Sigma} = 0 \text{ in } \Sigma, \text{ } \kappa_g = 0 \text{ on }
  \partial \Sigma . \label{intrisic rigid}
\end{equation}
We call a surface $\Sigma$ satisfying \eqref{interior rigid}, \eqref{boundary
corner rigid} and \eqref{intrisic rigid} {\itshape{infinitesimally rigid}}. Note that such
a surface is isometric to a flat $k$-polygon in $\mathbb{R}^2$.

\subsection{Foliation near infinitesimally rigid capillary surfaces}

Take a vector field $Y$ near $\Sigma$ such that $Y$ is tangential to $\partial
M$. Let $\phi (x, t)$ be the flow of $Y$.

\begin{proposition}\label{foliation away from vertex}
  Let $\Sigma$ be a properly embedded, two-sided, minimal capillary surface in
  $M$. If $\Sigma$ is infinitesimally rigid, then there exists $\varepsilon >
  0$ and a function $w : \Sigma \times (- \varepsilon, \varepsilon) \to
  \mathbb{R}$ such that for every $t \in (- \varepsilon, \varepsilon)$, the
  set
  \[ \Sigma_t = \{\phi (x, w (x, t)) : x \in \Sigma\} \]
  is a capillary surface with constant mean curvature $H (t)$ that meets $F_j$
  at constant angle $\bar{\gamma }_j$. Moreover, for every $x \in \Sigma$ and every
  $t \in (- \varepsilon, \varepsilon)$,
  \[ w (x, 0) = 0, \text{ } \int_{\Sigma} (w (x, t) - t) = 0 \text{ and }
     \tfrac{\partial}{\partial t} w (x, 0) = 1. \]
  By possibly choosing a smaller $\varepsilon$, $\{\Sigma_t \}_{t \in (-
  \varepsilon, \varepsilon)}$ is a foliation of a neighbourhood of $\Sigma_0 =
  \Sigma$ in M.
\end{proposition}

\subsection{Foliation near the vertex}\label{sec:foliation near vertex}
To obtain a foliation near the vertex is more delicate than Proposition \ref{foliation away from vertex}. Let $(M^3, g)$ be a cone type polyhedron with vertex $p$ with the flat reference $\bar{P}$, if the vertex is a minimiser to \eqref{action in terms of cacciopolli}, from the discussions at the beginning of Section \ref{cone type polyhedron} and Proposition \ref{prop_energy_est}, we conclude when $\bar{P}$ is a tetrahedron that
 the metric of the tangent cone of $M$ at $p$ is  conformal to $\delta$ (or a constant multiple of $\delta$). In particular, 
\begin{equation}
  \measuredangle{F_j F_{j + 1}} = \measuredangle{\bar{F}_j \bar{F}_{j + 1}} \label{angle condition}
\end{equation}
where $F_j$ are side faces of $M$, $F_j = \varphi (\bar{F}_j)$ and $\varphi$ is the diffeomorphism from the reference polyhedron $\bar{P}$ to $M$. 
For general cone polyhedra, this conformality is only true provided we confirmed Conjecture \ref{minimiser existence conjecture}. Note also because of Example \ref{tetrahedron example}, \eqref{angle condition} is a weaker condition than that tangent cone  $T_{\bar{p}}\bar{P}$ and the $T_pM$ are conformal.

But to obtain a foliation near the vertex, we can always assume that the tangent cone $T_pM$ is conformal to the $T_{\bar{p}}\bar{P}$ whether $\bar{P}$ is a tetrahedron or not. And after rescaling of $\bar{P}$, we can further assume they are isometric. We obtain the following.

\begin{theorem}
  \label{foliation near vertex}Let $(M^3, g)$ be a cone type Riemannian
  polyhedron with vertex $p$. Assume that there is a diffeomorphism
  \begin{equation}
    \varphi : \bar{P} \to M
  \end{equation}
  where $\bar{P} \subset \mathbb{R}^3$ is a cone type polyhedron with vertex $\bar{p} =
  \varphi^{- 1} (p)$, and the tangent cones $(T_p M, g_p)$ and $(T_{\bar{p}} \bar{P},
  \delta)$ are isometric. Denote $\bar{\gamma }_i$ be the angles between the
  base face and side faces of $\bar{P}$. Then there exists a small neighbourhood $U$
  of $p$ in $M$ such that $U$ is foliated by surfaces $\{\Sigma_{\rho}
  \}_{\rho \in (0, \varepsilon)}$ with the properties that
  \begin{enumerate}
    \item for each $\rho \in (0, \varepsilon)$, $\Sigma_{\rho}$ meet the side
    face $F_j$ at constant angle $\bar{\gamma }_j$;
    
    \item each $\Sigma_{\rho}$ has constant mean curvature $\lambda_{\rho}$.
  \end{enumerate}
\end{theorem}

\begin{remark}\label{li gap foliation}
  In general, there is no 
  diffeomorphism from $\bar{P}$ to $M$ near $p$ with the pullback metric $C^1$ close
  to the Euclidean metric. In other words, the geodesic normal coordinates do not work at the vertex for our purpose. So we use the coordinates provided by the diffeomorphism of $(M,g)$ to the reference polyhedron $\bar{P}$.
\end{remark}

\begin{proof}
  Let $\bar{\boldsymbol{C}}=T_{\bar{p}}\bar{P}$ be the tangent cone of $\bar{P}$ at its apex.
  Hence $\bar{\boldsymbol{C}}$ is a polyhedral cone in $\mathbb{R}^3$.
  Without loss of generality, we suppose $\bar{P}$ is given by
  \[
    \bar{P}=\{(x_1,x_2,x_3)\in \mathbb{R}^3:
    x_1\ge -t\}\quad \text{for some }t>0.
  \]
  
  Note that $\varphi$ already gives us a coordinate chart for $M$.
  Since $(T_pM, g_p)$ and $\bar{\boldsymbol{C}}$ are isometric, we know metric $g$ is the standard Euclidean metric $\delta$ at $p$.
  
  We define
  \[ \Sigma_{\rho} = \{x=(x_1,x_2,x_3) \in \bar{\boldsymbol{C}} : x_1=-\rho\} . \]
  Let $\phi (x, t) = x \mathrm{e}^t$ and
  \begin{equation*}
    \Sigma_{\rho, u} = \{\rho x \mathrm{e}^{u (x, \rho)} : x \in \Sigma_1 \} .
  \end{equation*}
  Let $y \in \Sigma_{\rho}$ and $\sigma_{i j}^{\rho} (x) = g_{i j} (\rho x)$
  on $\Sigma_1$. We write $\sigma = \sigma^{\rho}$ for convenience. Note that
  $x = \tfrac{y}{\rho} \in \Sigma_1$ and $\tfrac{\partial}{\partial y_i} =
  \tfrac{1}{\rho} \tfrac{\partial}{\partial x_i}$. By Taylor expansion of the
  mean curvature and the contact angle under metric $g$, we see by \eqref{first variation mean
  curvature} and \eqref{first variation of contact angle} that
\begin{align}
H_{\rho, u} - H_{\rho} & = - \tfrac{1}{\rho^2} \Delta_{\sigma} u -
(\ensuremath{\operatorname{Ric}}(N_{\rho}) + |A_{\rho} |^2) u + L_1 u +
Q_1 (u), \label{taylor mean} \\
\langle X_{\rho, u}, N_{\rho, u} \rangle - \langle X_{\rho}, N_{\rho}
\rangle & = - \tfrac{\sin \gamma_j}{\rho} \tfrac{\partial u}{\partial
\nu_{\sigma}} + (- \cos \gamma_j A (\nu_{\rho}, \nu_{\rho}) + A_{\partial M}(\tilde{\nu}_{\rho}, \tilde{\nu}_{\rho})) u \nonumber\\
& \quad + L_2 u + Q_2 (u) . \label{taylor angle}
\end{align}
  Here, $H_{\rho,u}$ ($H_\rho$ resp.) are the mean curvature of surface $\Sigma_{\rho,u}$ ($\Sigma_\rho$ resp.) under metric $g$, and
  
  \begin{align*}
  \Delta_{\sigma} u : ={}& \sum_{i \neq 1, j \neq 1} \tfrac{1}{\sqrt{\det
    \sigma (x)}} \tfrac{\partial}{\partial x_i} (\sqrt{\det \sigma (x)}
    \sigma^{i j} (x) \tfrac{\partial}{\partial x_j} u (x))\\
 \nu_{\sigma} (x) := {}&\sum_{i \neq 1} \nu_{\rho}^i (\rho x)
    \tfrac{\partial}{\partial x^i},
  \end{align*}
  and \eqref{taylor mean} and \eqref{taylor angle} are valid since we see $u$
  as a function of $y$, $\tfrac{1}{\rho^2} \Delta_{\sigma} u = \Delta_{\rho}
  u$ and $\tfrac{1}{\rho} \tfrac{\partial u}{\partial \nu_{\sigma}} = \sum_{i
  \neq 1} \nu_{\rho}^i (y) \tfrac{\partial u}{\partial y^i}$.
  
  The functions $L_1$ and $L_2$ measure how the mean curvature and the contact
  angle deviate from being constant. They satisfy the bounds
  \begin{align*}
      |L_1 | \leq {}&| \nabla_{\rho} H_{\rho} |  |Y| \leq C |g|_{C^2}
    \leq C,\\
    |L_2 | \leq {}&| \nabla_{\rho} \langle X_{\rho}, N_{\rho} \rangle |  |Y|
    \leq C |g|_{C^1} \leq C.
  \end{align*}
  Letting $\rho \to 0$, we see that $\Delta_{\sigma}$ and $\nu_{\sigma}$
  converges to Euclidean Laplacian and Euclidean normal which we denote by
  $\Delta$ and $\nu$ respectively.
  
  Denote $D_{\rho} = \langle X_{\rho}, N_{\rho} \rangle - \cos
  \bar{\gamma }_j=\cos \gamma_j-\cos \bar{\gamma}_j$. We define the operators
\begin{align*}
\mathcal{L}_{\sigma} u & = - \Delta_{\sigma} u + \rho^2 [-
(\ensuremath{\operatorname{Ric}}(N_{\rho}) + |A_{\rho} |^2) u + L_1 u +
Q_1 (u)] + \rho^2 H_{\rho}, \\
\mathcal{B}_{\sigma} u & = \sin \gamma_j \tfrac{\partial u}{\partial
\nu_{\sigma}} - \rho D_{\rho} - \rho [(- \cos \gamma_j A (\nu_{\rho},
\nu_{\rho}) + A_{\partial M}(\tilde{\nu}_{\rho}, \tilde{\nu}_{\rho})) u + L_2 u + Q_2
(u)],
\end{align*}
  and the Banach spaces
\begin{align*}
\mathcal{X} & = \{u \in C^{2, \alpha} (\Sigma_1) \cap C^{1, \alpha}
(\bar{\Sigma}_1)\}, \\
\mathcal{Y} & = \{u \in C^{\alpha} (\Sigma_1) : \int_{\Sigma_1} u = 0\},
\\
\mathcal{Z} & = \{u \in L^{\infty} (\partial \Sigma_1) : u_{| \bar{L}_j}
\in C^{0, \alpha} (\bar{L}_j)\} .
\end{align*}
  Here the integration is with respect to the limiting metric,
  \[ \lim_{\rho \to 0} \sigma^{\rho} = \lim_{\rho \to 0} g_{i j} (\rho x) =
     \sigma^0 \]
  and the limiting metric is really the metric of the tangent cone at $p \in
  M$.
  
  For a small $\delta > 0$, let $\Psi : (- \varepsilon, \varepsilon) \times
  B_{\delta} (0) \to \mathcal{Y} \times \mathcal{Z}$
  \[ \Psi (\rho, u) = \left(\mathcal{L}_{\sigma} u - \tfrac{1}{| \Sigma_1 |}
     \int_{\Sigma_1} \mathcal{L}_{\sigma} (u), \mathcal{B}_{\sigma} u\right), \]
  where $B_{\delta} (0) \subset \mathcal{X}$ is an open neighborhood of the
  zero function. The linearized operator
  \[ D \Psi_{| (0, 0)} (0, v) = \tfrac{\mathrm{d}}{\mathrm{d} s}_{|s = 0} \Psi
     (0, s v) . \]
  By the limiting behaviour of $\mathcal{L}_{\sigma}$ and
  $\mathcal{B}_{\sigma}$ as $\rho \to 0$, we find that
  \[ D \Psi_{| (0, 0)} (0, v) = \left(\Delta v - \int_{\Sigma_1} \Delta v, \sin
     \bar{\gamma}_j \tfrac{\partial v}{\partial \nu}\right) . \]
  Following from {\cite{lieberman-optimal-1989}} (or {\cite[Theorem
  4.2]{li-polyhedron-2020}}), we see that $D \Psi_{| (0, 0)}$ is an isomorphism
  when restricted to $\{0\} \times \mathcal{X}$. We, therefore, apply the
  inverse function theorem and conclude that for sufficiently small
  $\varepsilon > 0$, there exists a $C^1$ map from $(- \varepsilon,
  \varepsilon)$ to $B_{\delta} (0) \subset \mathcal{X}$ such that
  \[ \Psi (\rho, u (\cdot, \rho)) = 0,\quad u (\cdot, \rho) \in B_{\delta} (0) \]
  for every $\rho \in (- \varepsilon, \varepsilon)$. Now we fix such $u (\cdot
  {,} \rho)$, thus $\Sigma_{\rho, u (\cdot, \rho)}$ is of constant mean
  curvature and meets each $F_j$ at constant angles $\bar{\gamma}_j$ by
  \eqref{taylor mean} and \eqref{taylor angle}.
  
  Let $\lambda_{\rho} = H_{\rho, u (\cdot, \rho)}$ we obtain the theorem.
  Indeed, by the construction of $\Sigma_{\rho, u (\cdot, \rho)}$ we have
\begin{align}
- \Delta_{\sigma} u + \rho^2 [- (\ensuremath{\operatorname{Ric}}(N_{\rho})
+ |A_{\rho} |^2) u + L_1 u + Q_1 (u) + H_{\rho} - \lambda_{\rho}] & = 0
\text{ in } \Sigma_1, \label{laplace} \\
\sin \gamma_j \tfrac{\partial u}{\partial \nu_{\sigma}} - \rho D_{\rho} -
\rho [- \cos \gamma_j A (\nu_{\rho}, \nu_{\rho}) u + {A_{\partial M}}
(\tilde{\nu}_{\rho}, \tilde{\nu}_{\rho}) u + L_2 u + Q_2 u] & = 0 \text{ on }
\partial \Sigma_1 . \label{normal derivative }
\end{align}
  By definition $u (\cdot, 0)$ is the zero function. Let $v = \tfrac{\partial
  u (\cdot, \rho)}{\partial \rho}$. Differentiating the above with respect to
  to $\rho$ and evaluating at $\rho = 0$, we deduce that
  \[ \Delta v = 0 \text{ in } \Sigma_1, \tfrac{\partial v}{\partial \nu} = 0
     \text{ on } \partial \Sigma_1 . \]
  So $v$ is a constant. By $u (\cdot, \rho) \in \mathcal{X}$, we have that $v
  \in \mathcal{X}$ as well, so $v \equiv 0$ on $\Sigma_1$. So we conclude
  \begin{equation}
    |u (\cdot, \rho) |_{C^{1, \alpha} (\bar{\Sigma}_1)} = o (\rho)
    \label{order of u}
  \end{equation}
  for $\rho$ small by elliptic regularity of {\cite{lieberman-optimal-1989}}
  (or {\cite[Theorem 4.2]{li-polyhedron-2020}}). Therefore, $\Sigma_{\rho, u
  (\cdot, \rho)}$ is a foliation of a small neighbourhood of $p$.
\end{proof}

\begin{lemma}
  \label{mean curvature bound near vertex}Let $\Sigma_{\rho, u (\cdot, \rho)}$
  and $\lambda_{\rho}$ be constructed above, then
  \[ \lim_{\rho \to 0} \lambda_{\rho} \geq - 2. \]
\end{lemma}

\begin{proof}
  We integrate \eqref{laplace} on $\Sigma_1$ with respect to the metric
  $\sigma^{\rho}$ we find the constant mean curvature $\lambda_{\rho}$ of
  $\Sigma_{\rho, u (\cdot, \rho)}$ satisfies
\begin{align*}
\lambda_{\rho} | \Sigma_1 |_{\sigma} = & - \tfrac{1}{\rho^2}
\int_{\Sigma_1} \Delta_{\sigma} u + \int_{\Sigma_1} L_1 u + Q_1 (u) +
H_{\rho} -\ensuremath{\operatorname{Ric}} (N_{\rho}, N_{\rho}) u -
|A_{\rho} |^2 u \\
= & - \tfrac{1}{\rho^2} \int_{\partial \Sigma_1} \tfrac{\partial
u}{\partial \nu_{\sigma}} + \int_{\Sigma_1} L_1 u + Q_1 (u) + H_{\rho}
-\ensuremath{\operatorname{Ric}} (N_{\rho}, N_{\rho}) u - |A_{\rho} |^2 u
\\
= & - \tfrac{1}{\rho} \int_{\partial \Sigma_1} \tfrac{1}{\sin \gamma_j}
(D_{\rho} - \cos \gamma_j A (\nu_{\rho}, \nu_{\rho})u + A_{\partial M}(\tilde{\nu}_{\rho}, \tilde{\nu}_{\rho}) u + L_2 u + Q_2 u) \\
& \quad + \int_{\Sigma_1} L_1 u + Q_1 (u) + H_{\rho}
-\ensuremath{\operatorname{Ric}} (N_{\rho}) u - |A_{\rho} |^2 u.
\end{align*}
  Here $| \Sigma_1 |_{\sigma}$ denotes the area of $\Sigma_1$ under the metric
  $\sigma^{\rho}$. Since $|u|_{C^{1, \alpha} (\bar{\Sigma}_1)} = o (\rho)$,
  and $\sigma$ is closed to the constant metric, we see that
  \[ \lambda_{\rho} | \Sigma_1 |_{\sigma^0} + o (1) = - \tfrac{1}{\rho}
     \int_{\partial \Sigma_1} \tfrac{1}{\sin \gamma_j} D_{\rho} +
     \int_{\Sigma_1} H_{\rho} . \]
  Since
\begin{align*}
D_{\rho} = & \cos \gamma_j - \cos \bar{\gamma }_j \\
= & - \sin \gamma_j (\gamma_j - \bar{\gamma }_j) +
O (| \gamma_j -
\bar{\gamma }_j |^2), \\
= & - \sin \gamma_j (\gamma_j - \bar{\gamma }_j) + O (\rho^2),
\end{align*}
  so
  \begin{equation}
    \lambda_{\rho} + o (1) = \tfrac{1}{\rho | \Sigma_1 |_{\sigma^0}} \left(
    \int_{\partial \Sigma_1} (\gamma_j - \bar{\gamma }_j) + \int_{\Sigma_1}
    H_{\rho} \right) . \label{asymptotic of lambda rho}
  \end{equation}
  Now we take integration with respect to $\sigma^0$ (still omitting the area
  element and line element). In fact, the right-hand side of the above differs
  by only $o (1)$ computed with respect to $\sigma^{\rho}$ or $\sigma^0$.
  
  We calculate the mean curvature $H_{\rho}$ of $\Sigma_{\rho}$. Using $y$
  coordinate, we have at $y$
  \[ H_{\rho} = - \tfrac{1}{2} h^{a b} (g^{1 1})^{- \tfrac{1}{2}} g^{1 k} (2
     \tfrac{\partial}{\partial y^b} g_{a k} - \tfrac{\partial}{\partial y^k}
     g_{a b, k}) \]
  where $a$, $b$ ranges from 2 to 3 and $h^{a b}$ is the inverse metric of
  $g_{a b}$ (as a $2 \times 2$ matrix). By changing the coordinate to $x$, we
  see at $x = \tfrac{y}{\rho} \in \Sigma_1$,
\begin{align*}
H_{\rho} = & - \tfrac{1}{2} h^{a b} (g^{1 1}) g^{1 k} (2 \tfrac{1}{\rho}
\tfrac{\partial}{\partial x^b} g_{a k} (\rho x) - \tfrac{1}{\rho}
\tfrac{\partial}{\partial x^k} g_{a b, k} (\rho x)) \\
= & - \tfrac{1}{\rho} [h^{a b} (g^{1 1}) g^{1 k} (2
\tfrac{\partial}{\partial x^b} g_{a k} (\rho x) -
\tfrac{\partial}{\partial x^k} g_{a b, k} (\rho x))] .
\end{align*}
  Let $\Omega_1$ be the portion of polyhedron on the upper side of $\Sigma_1$.
  We know that $\sigma^{\rho}_{i j} (x) = g_{i j} (\rho x)$, $x \in \Omega_1$
  is a family of metrics indexed by $\rho$ on $\Omega_1$ converging to the
  constant metric $\delta_{i j}$. We have by Taylor expansion,
  \[ \sigma_{i j}^{\rho} (x) = g_{i j} (\rho x) = \delta_{i j} + \sum_k \rho
     x^k \tfrac{\partial}{\partial y^k} g_{i j} (0) + O (\rho^2) . \]
  Now the mean curvature $H_{\rho}$ is just $\tfrac{1}{\rho} H_{\sigma}$ at
  $x$. The $H_{\sigma}$ is the mean curvature of $\Sigma_1$ at $x$ computed
  using the metric $\sigma^{\rho}$ and with respect to the normal pointing
  inside of $\Omega_1$. As $\rho \to 0$, $H_{\sigma} \to 0$. We can think of
  $H_{\rho}$ as the first variation of $H_{\sigma}$ under the perturbation
  $\delta \mapsto \delta + \sum_k \rho x^k \tfrac{\partial}{\partial y^k} g_{i
  j} (0)$. Similar analysis applies to $\tfrac{1}{\rho} \int_{\partial
  \Sigma_1} (\gamma_j - \bar{\gamma }_j)$.
  
  To sum up, the quantity
  \begin{equation}
    \tfrac{1}{\rho} \left( \int_{\partial \Sigma_1} (\gamma_j -
    \bar{\gamma }_j) + \int_{\Sigma_1} H_{\sigma} \right)\label{base face}
  \end{equation}
  as $\rho \to 0$ is the first variation of $\int_{\partial \Sigma_1}
  (\gamma_j - \bar{\gamma }_j) + \int_{\Sigma_1} H_{\sigma}$.
  
  Let $R_{\sigma}$ be the scalar curvature of $\Omega_1$ computed with respect
  to $\sigma$, and $H_{\sigma}$ is the mean curvature of the faces $\partial
  \Omega_1$ computed with respect to the normal pointing {\underline{inside}}
  of $\Omega_1$. We claim that the quantity
  \begin{equation}
    \int_{\Omega_1} R_{\sigma} - 2 \int_{\partial \Omega_1} H_{\sigma} - 2
    \int_{E (\Omega_1)} (\gamma - \bar{\gamma }) = O (\rho^2) \label{local
    miao-piubello}
  \end{equation}
  as $\rho \to 0$. Here $E(\Omega_1)$ denotes the union of all edges of $\Omega_1$.
  Indeed, \eqref{local miao-piubello} is a local version of
  {\cite[(3.33)]{miao-mass-2021}}. Let $h = x^k \tfrac{\partial}{\partial y^k}
  g_{i j} (0) + O (\rho)$. We only have to note by {\cite[Proposition
  4]{brendle-scalar-2011}} that
  \[ R_{\sigma} = - \rho \tfrac{\partial}{\partial x^i} (\partial_i h_{j j} -
     \partial_j h_{i j}) + O (\rho^2) . \]
  Note that the term of order $O (1)$ of $h$ is linear in $x$, so in fact,
  \[ R_{\sigma} = O (\rho^2) . \]
  By divergence theorem
  \[ \int_{\partial \Omega_1} R_{\sigma} = - \rho \int_{\partial \Omega_1}
     (\tfrac{\partial}{\partial x^i} h_{j j} - \tfrac{\partial}{\partial x^j}
     h_{i j}) \tilde{\nu}^i + O (\rho^2) = O (\rho^2) \]
  where $\tilde{\nu}$ are the unit normal of $\partial \Omega_1$ pointing
  outward of $\Omega_1$. The rest is formally the same with
  {\cite{miao-mass-2021}}.
  
  Now using \eqref{local miao-piubello}, we can find the value of \eqref{base
  face}. So
\begin{align*}
& \tfrac{1}{\rho} \left( \int_{\partial \Sigma_1} (\gamma_j -
\bar{\gamma }_j) + \int_{\Sigma_1} H_{\sigma} \right) \\
= & - \tfrac{1}{\rho} \left( \int_{\partial \Omega_1 \backslash \Sigma_1}
H_{\sigma} + \int_{E (\Omega_1)\backslash \partial \Sigma_1} (\gamma -
\bar{\gamma }) \right) + O (\rho) .
\end{align*}
  Using the assumptions on $H_{\rho}$, we have that $H_{\sigma} = \rho
  H_{\rho} \leq - \rho H_{\partial \Omega_1}$ on $\partial \Omega_1\backslash \Sigma_1$. And also $\gamma \leq
  \bar{\gamma }$, so
  \[ \tfrac{1}{\rho} \left( \int_{\partial \Sigma_1} (\gamma_j -
     \bar{\gamma }_j) + \int_{\Sigma_1} H_{\sigma} \right) \geq \left(
     \int_{\partial \Omega_1 \backslash \Sigma_1} H_{\partial \Omega_1} \right) + O (\rho),
  \]
  And we have the bound that $H_{\partial \Omega_1} \geq - 2 \cos \bar{\gamma }_j$ on
  each reference face $\bar{F}_j$, so by the above and \eqref{asymptotic of lambda
  rho},
  \begin{equation}
    \lambda_{\rho} \geq \tfrac{- 2}{| \Sigma_1 |_{\sigma^0}}
    \int_{\partial \Omega_1 \backslash \Sigma_1} \cos \bar{\gamma }_j + o (1) .
  \end{equation}
We pick one side
  face $\bar{F}_j$, and $\bar{F}_j$ is a triangle. The product of its area and $\cos
  \bar{\gamma }_j$ is the signed area of the projected triangle to the plane in which
  $\Sigma_1$ lies in. Summing them up overall side faces gives the area $|
  \Sigma_1 |_{\sigma^0}$, so we know that
  \begin{equation}
    \int_{\partial \Omega_1 \backslash \Sigma_1} \cos \bar{\gamma }_j = |
    \Sigma_1 |_{\sigma^0}, \label{area projection}
  \end{equation}
  therefore $\lim_{\rho \to 0} \lambda_{\rho} \geq - 2$.
\end{proof}

\begin{remark}
	It is possible that $\lim_{\rho \to 0} \lambda_{\rho} > - 2$.
\end{remark}

\subsection{Local splitting}

Assume that we have a local CMC foliation $\{\Sigma_{\rho} \}_{\rho \in I}$
where as $\rho$ increases $\Sigma_{\rho}$ moves in the direction of
$N_{\rho}$. We take $I$ to be $(- \varepsilon, \varepsilon)$, $(-
\varepsilon, 0)$ or $(0, \varepsilon)$ according to the location of the
foliation. 
Note that we have $\gamma_j=\bar{\gamma}_j$ for each leaf $\Sigma_\rho$ of this foliation.

\begin{proposition}
  \label{mean curvature ode}There exists a nonnegative continuous function $C
  (\rho)$ such that
  \begin{equation}
    - H' (\rho) \geq - C (\rho) (H (\rho) + 2) .
  \end{equation}
\end{proposition}

\begin{proof}
  Let $\psi : \Sigma \times I \to M$ parametrizes the foliation. Denote $Y =
  \tfrac{\partial \psi}{\partial t}$, $v_{\rho} = \langle Y, N_{\rho}
  \rangle$. Then
  \[ - \tfrac{\mathrm{d}}{\mathrm{d} \rho} H (\rho) = \Delta_{\rho} v_{\rho} +
     (\ensuremath{\operatorname{Ric}}(N_{\rho}) + |A_{\rho} |^2) v_{\rho}
     \text{ in } \Sigma_{\rho}, \label{derivative of H} \]
  and
  \[ \tfrac{\partial v_{\rho}}{\partial \nu_{\rho}} = [(- \cot \gamma_j
     A_{\rho} (\nu_{\rho}, \nu_{\rho}) + \tfrac{1}{\sin \gamma_j} A_{\partial M}(\tilde{\nu}_{\rho}, \tilde{\nu}_{\rho}))] v_{\rho} \text{ on } \partial
     \Sigma_{\rho} \cap F_j . \]
  By shrinking the interval $I$ if needed, we can assume that $v_{\rho} > 0$
  for $\rho \in I$. By multiplying $\tfrac{1}{v_{\rho}}$ on both sides of
  \eqref{derivative of H} and integrating on $\Sigma_{\rho}$, we deduce that
\begin{align*}
& - H' (\rho) \int_{\Sigma_{\rho}} \tfrac{1}{v_{\rho}} \\
= & \int_{\Sigma_{\rho}} \tfrac{\Delta_{\rho} v_{\rho}}{v_{\rho}} +
(\ensuremath{\operatorname{Ric}}(N_{\rho}) + |A_{\rho} |^2) \\
= & \int_{\partial \Sigma_{\rho}} \tfrac{1}{v_{\rho}} \tfrac{\partial
v_{\rho}}{\partial \nu_{\rho}} + \tfrac{1}{2} \int_{\Sigma_{\rho}} (R_g +
|A_{\rho} |^2 - H^2_{\rho}) - \int_{\Sigma_{\rho}} K_{\Sigma_{\rho}}
\\
\geq & \sum_j \int_{\partial \Sigma_{\rho} \cap F_j} [(- \cot
\gamma_j A_{\rho} (\nu_{\rho}, \nu_{\rho}) + \tfrac{1}{\sin \gamma_j} A_{\partial M}(\tilde{\nu}_{\rho}, \tilde{\nu}_{\rho}))] - \int_{\Sigma_{\rho}}
K_{\Sigma_{\rho}} .
\end{align*}
  Using the Gauss-Bonnet theorem and {\cite[Lemma 3.2]{li-polyhedron-2020}},
  \begin{equation*}
    - \int_{\Sigma_{\rho}} K_{\Sigma_{\rho}} \geq \int_{\partial
    \Sigma_{\rho}} \kappa_g .
  \end{equation*}
  As in \eqref{hyperoblic contribution of the boundary},
  \[ \kappa_g - \cot \gamma_j A (\nu_{\rho}, \nu_{\rho}) + \tfrac{1}{\sin
     \gamma_j} A_{\partial M}(\tilde{\nu}_{\rho}, \tilde{\nu}_{\rho}) = - H (\rho) \cot
     \gamma_j + \tfrac{1}{\sin \gamma_j} \bar{H} . \]
  So
\begin{align*}
& - H' (\rho) \int_{\Sigma_{\rho}} \tfrac{1}{v_{\rho}} \\
\geq & \sum_j \int_{\partial \Sigma_{\rho}} [- H (\rho) \cot \gamma_j
+ \tfrac{1}{\sin \gamma_j} \bar{H}] \\
\geq & \sum_j \int_{\partial \Sigma_{\rho}} [- H (\rho) \cot \gamma_j
- \tfrac{2}{\sin \gamma_j} \cos \gamma_j] \\
= & \sum_j \cot \gamma_j [- H (\rho) - 2] | \partial \Sigma_{\rho} \cap
F_j |,
\end{align*}
  where we have used $H_j \geq - 2 \cos \gamma_j$. We take
  \[ C (\rho) = \sum_j \cot \gamma_j  | \partial \Sigma_{\rho} \cap F_j |. \]
  It is clear that $\gamma_j \leq \frac{\pi}{2}$ obvious implies that $C(\rho)\geq 0$. When the reference satisfies the cylinder condition, we know  from the following Lemma \ref{lem_bdry_infi_rig} that $C(0)>0$, by continuity, $C(\rho)>0$ for small $|\rho|$. So the proposition is proved.
\end{proof}

\begin{lemma}
  \label{lem_bdry_infi_rig} 
  Let $\bar{P}$ be a flat reference polyhedron satisfying
  the cylinder trapping condition (see Definition \ref{cylinder condition}) in the solid cylinder $\bar{\mathcal{C}}$, and $\bar{B}$
  be the face such that $\partial \bar{B} \subset \partial \bar{\mathcal{C}}$. Let
  $\bar{\gamma}_j$ be the dihedral angle of the side face $\bar{F}_j$ and the base face
  $\bar{B}$, then
  \begin{equation}
    \sum_j^{} \cot \bar{\gamma}_j  | \partial \bar{B} \cap \bar{F}_j | > 0. \label{sign of C rho}
  \end{equation}
\end{lemma}

\begin{proof}
  Let $\hat{F}_j$ be the face of $\bar{\mathcal{C}}$ passing through $\partial \bar{B}
  \cap \bar{F}_j$ and $\hat{\gamma}_j$ be the dihedral angle formed by
  $\hat{F}_j$ and $\bar{B}$. From the cylinder trapping condition implies that
  \begin{equation*}
    \hat{\gamma}_j \geq\bar{\gamma }_j
  \end{equation*}
  for all $j$ and for at least one $j$, the inequality is strict, so
  \begin{equation*}
    \sum_j^{} \cot \bar{\gamma }_j  | \partial \bar{B} \cap \bar{F}_j | > \sum_j \cot
    \hat{\gamma}_j | \partial \bar{B} \cap \hat{F}_j | .
  \end{equation*}
  It suffices to show that the right-hand side vanishes.
  
  Let $\vec{v}$ be the vector parallel to the edges of $\bar{\mathcal{C}}$ pointing
  to the upside of $\bar{B}$. Let $p_j$ be the vertices of the polygon $\bar{B}$ such
  that $p_1 p_2 \ldots p_j$ are arranged in the counter-clockwise direction,
  so the unit normal pointing outward of $\bar{\mathcal{C}}$ is
  \[ N_j = \frac{\overrightarrow{p_j p_{j + 1}} \times
     \vec{v}}{|\overrightarrow{p_j p_{j + 1}} \times \vec{v|}} . \]
  Let $W$ be the normal of $\bar{B}$, so
  \[ \cos \hat{\gamma}_j = N_j \cdot W, \sin \hat{\gamma}_j = |W \times N_j |
     . \]
  So
  \[ \cot \hat{\gamma}_j = \frac{N_j \cdot N}{|N \times N_j |} =
     \frac{(\overrightarrow{p_j p_{j + 1}} \times \vec{v}) \cdot W}{|
     (\overrightarrow{p_j p_{j + 1}} \times \vec{v}) \times W|} =
     \frac{(\overrightarrow{p_j p_{j + 1}} \times \vec{v}) \cdot W}{|
     \overrightarrow{p_j p_{j + 1}} (\vec{v} \cdot W) |}, \]
  where we have used the cross product formula $(\vec{a} \times \vec{b})
  \times \vec{c} = \vec{b} (\vec{a} \cdot \vec{c}) - \vec{a}  (\vec{b} \cdot
  \vec{c})$ for vectors $\vec{a}$, $\vec{b}$, $\vec{c}$ in $\mathbb{R}^3$. So
  
  \begin{align}
    \sum_j \cot \hat{\gamma}_j | \partial \bar{B} \cap \hat{F}_j | = & \sum_j
    \frac{(\overrightarrow{p_j p_{j + 1}} \times \vec{v}) \cdot W}{|
    \overrightarrow{p_j p_{j + 1}} (\vec{v} \cdot N) |} | \overrightarrow{p_j
    p_{j + 1}} | \nonumber\\
    = & \frac{1}{\vec{v} \cdot W} \sum_j (\overrightarrow{p_j p_{j + 1}}
    \times \vec{v}) \cdot W \nonumber\\
    = & \frac{1}{\vec{v} \cdot W} \left[ \left( \sum_j \overrightarrow{p_j
    p_{j + 1}} \right) \times \vec{v} \right] \cdot W = 0. \nonumber
  \end{align}
  
  \ 
\end{proof}

\begin{remark}
  \label{equivalent cylinder trapping}Note that the cylinder $\bar{\mathcal{C}}$
  is uniquely determined by the vector $\vec{v}$. So the cylinder trapping
  condition is equivalent to the existence of a constant vector $\vec{v}$ such
  that
  \begin{equation}
    N \cdot \frac{\overrightarrow{p_j p_{j + 1}} \times \vec{v}}{|
    \overrightarrow{p_j p_{j + 1}} \times \vec{v} |} \leq \cos \bar{\gamma }_j,
  \end{equation}
  for all $j$ and for at least one $j$ the above inequality is strict. 
\end{remark}

\begin{remark}
  The length of the edges of $\bar{B}$ plays no role in the proof of Lemma
  \ref{lem_bdry_infi_rig}. If $\bar{P}$ is cylinder trapped, for another polygon
  $\bar{B}_1$ with the same interior angles as $\bar{B}$, the inequality \eqref{sign of C
  rho} is also true for $\bar{B}_1$. This can be also seen from Remark
  \ref{equivalent cylinder trapping}.
\end{remark}

\begin{lemma}
  If $\bar{P}$ is a flat tetrahedron reference, \eqref{sign of C rho} is
  automatically true.
  \label{lem_tetrahedron_trapping}
\end{lemma}

\begin{proof}
  We note that $\bar{B}$ is a triangle. Let $q$ be the projection of the vertex $p$
  to the plane of the base face $\bar{B}$. We label tree vertices of the triangle
  $p_1$, $p_2$ and $p_3$, denote by $l_j$ the length of the edge opposite to
  $p_i$. Let the straight line passing through $p_1$ and $p_2$ be $\ell_3$,
  $p_4$ be the point on $\ell_3$ realizing the distance from $q$ to $\ell$. So
  \[ \cot \bar{\gamma }_j = \pm \tfrac{|qp_4 |}{|pq|} . \]
  The sign depends on whether $q$ lies on the same side of $\ell_3$ as $p_3$. If
  it lies on the same side, then it is a positive sign. And
  \[ |qp_4 | = \tfrac{\mathrm{Area} (\triangle p_1 p_2 q)}{|p_1 p_2 |} . \]
  So we see that
  \begin{equation*}
    |p_1 p_2 | \cot \bar{\gamma }_j = \pm \tfrac{\mathrm{Area} (\triangle p_1
    p_2 q)}{|pq|} .
  \end{equation*}
  Summing over three edges of $\triangle p_1 p_2 p_3$, we see that
  \[ \sum_{j = 1}^3 \cot \bar{\gamma }_j l_j = \tfrac{\mathrm{Area} (\triangle
     p_1 p_2 p_3)}{|pq|}, \]
  which is positive.
\end{proof}
\begin{remark}
The proof actually works for any flat cone.
\end{remark}
\

Now we prove Theorem \ref{rigidity}.

\begin{proof}[Proof of Theorem \ref{rigidity}]
By Theorem \ref{thm_regularity},
  if $(M^3, g)$ is of prism type or if $(M^3, g)$ is of cone type with $I <
  0$, then the variational problem \eqref{action} has a nontrivial solution
  $E$ with a $C^{1, \alpha}$ boundary $\Sigma$. Therefore $\Sigma$ is
  an infinitesimally rigid capillary surface of mean curvature $- 2$, and there is a
  CMC capillary foliation $\{\Sigma_{\rho} \}_{\rho \in I}$ around $\Sigma$
  where $I = (- \varepsilon, \varepsilon)$ if $\Sigma$ lies in the interior of $M$, $I = [0, \varepsilon)$ if
  $\Sigma = \bar{B}_1$, and $I = (- \varepsilon, 0]$ if $\Sigma = \bar{B}_2$. By
  Proposition \ref{mean curvature ode}, the mean curvature of $H (\rho)$ of
  $\Sigma_{\rho}$ satisfies
  \begin{equation}
    - (H (\rho) + 2)' \geq - C (\rho) (H (\rho) + 2),
  \end{equation}
  and $H (0) = - 2$ where $C (\rho) > 0$. By standard ordinary differential
  equation theory,
  \begin{equation}
    H (\rho) \leq - 2 \text{ when } \rho \geq 0, \text{ } H (\rho)
    \geq - 2 \text{ when } \rho \leq 0.
  \end{equation}
  Denote $E_{\rho}$ be the corresponding open domain in $M$. Since
  $\Sigma_{\rho}$ meets $F_j$ at constant angles $\bar{\gamma }_j$, the first
  variation \eqref{first variation of action} formula implies that
  \begin{equation*}
    F (\rho_1) - F (\rho_2) = \int_{\rho_1}^{\rho_2} \mathrm{d} \rho
    \int_{\Sigma_{\rho}} [H (\rho) + 2] v_{\rho} .
  \end{equation*}
  For $0 < \varepsilon_1 < \varepsilon$,
  \begin{equation*}
    F (\varepsilon_1) \leq F (0), \text{ } F (- \varepsilon_1) \leq
    F (0) .
  \end{equation*}
  By since $\Sigma_0 = \Sigma$ is a minimiser of the functional
  \eqref{action}. Therefore in a neighborhood of $\Sigma$, $F (\rho) = F (0)$
  and $H (\rho) = - 2$. Tracing back to equality conditions, we find that
  $v_{\rho}$ is constant and each $\Sigma_{\rho}$ is infinitesimally rigid.
  
  Now we show that the vector field $Y^{\bot}$ is conformal. Because each
  $\Sigma_{\rho}$ is umbilic and of constant mean curvature $- 2$, so
  $\nabla_{\partial_i} N = - \partial_i$. Since $u = \langle Y, N \rangle$ is
  constant on each $\Sigma_{\rho}$, so
\begin{align*}
0 & = \nabla_{\partial_i} \langle Y, N \rangle \\
& = \langle \nabla_{\partial_i} Y, N  \rangle + \langle Y,
\nabla_{\partial_i} N \rangle \\
& = \langle \nabla_Y \partial_i, N \rangle - \langle Y, \partial_i
\rangle \\
& = Y \langle \partial_i, N \rangle - \langle \nabla_Y N, \partial_i
\rangle - \langle Y, \partial_i \rangle .
\end{align*}
  Observe that
  \[ \nabla_Y N = \nabla_{Y^{\top}} N + \nabla_{Y^{\bot}} N = - Y^{\top} +
     \nabla_{Y^{\bot}} N, \langle Y, \partial_i \rangle = Y^{\bot}, \]
  hence
  \[ \langle \nabla_{Y^{\bot}} N, \partial_i \rangle = 0. \]
  Moreover, we have using umbilicity, $\nabla_{Y^{\bot}} \langle \partial_i,
  \partial_j \rangle = - 2 u g_{i j}$. These conditions ensure that the
  foliation formed a subset $\cup_{\rho} \Sigma_{\rho}$ of hyperbolic 3-space
  and it is a local splitting. Since $M$ is connected, and considering that
  each leaf is of constant contact angle with side faces, we conclude that $M$
  is isometric to a polyhedron in the upper half space model.
  
  If $M$ is of cone type with the vertex as the minimiser, then by Theorem
  \ref{foliation near vertex} and Lemma \ref{mean curvature bound near
  vertex}, there is a CMC capillary foliation $\{\Sigma_{\rho} \}_{\rho \in (-
  \varepsilon, 0)}$ near the vertex with $\lim_{\rho \to 0} H (\rho) \geq
  - 2$. So we have either  $\lim_{\rho \to 0} \lambda_{\rho} > - 2$ or
  $\lim_{\rho \to 0} \lambda_{\rho} = - 2$.
  
  If $\lim_{\rho \to 0} \lambda_{\rho} > - 2$. By continuity, for all small
  $\rho \in (- \varepsilon, 0]$, $\lambda_{\rho} > - 2$. For some
  $\Sigma_{\rho_0, u (\cdot, \rho_0)}$ where $\rho_0 \in (0, \varepsilon)$,
  since the contact angle is constantly $\bar{\gamma }_j$ along each side edges
  of $\partial \Sigma_{\rho_0}$ by the construction of $\Sigma_{\rho, u
  (\cdot, \rho)}$, $\Sigma_{\rho, u (\cdot, \rho)}$ with $\rho_0 < \rho < 0$
  is a deformation of $\Sigma_{\rho_0, u (\cdot, \rho_0)}$ moving toward the
  vertex. Because $\lambda_{\rho} > - 2$, by the first variation,
  \begin{equation}
    F (\rho) - F (\rho_0) = \int_{\rho_0}^{\rho} \mathrm{d} \rho
    \int_{\Sigma_{\rho}} (H (\rho) + 2) v_{\rho} \label{first variation near
    vertex}
  \end{equation}
  when $- \varepsilon < \rho_0 < \rho < 0$. So $F (\rho)$ is a strictly
  increasing as $\rho \nearrow 0$, so
  \begin{equation*}
    0 > F (\Sigma_{\rho_0, u (\cdot, \rho_0)})
  \end{equation*}
  contradicts that the vertex is a minimiser.
  
  If $\lim_{\rho \to 0} \lambda_{\rho} = - 2$, we have that $H (\rho)
  \geq - 2$ for $\rho \in (- \varepsilon, 0) $ by Proposition \ref{mean
  curvature ode}. Let $E_{\rho}$ be the bounded subset bounded by
  $\Sigma_{\rho}$, we still have \eqref{first variation near vertex}, we have
  that
  \begin{equation*}
    F (\rho) \geq F (\rho_0) .
  \end{equation*}
  By letting $\rho \to 0$, we have that
  \begin{equation*}
    0 \geq F (\rho_0) .
  \end{equation*}
  So we conclude that $F (\rho) \equiv 0$ for $\rho \in (- \varepsilon, 0)$
  and that each leaf $\Sigma_{\rho}$ is infinitesimally rigid. Thus as before,
  $(M^3, g)$ admits a global splitting of flat polygons in $\mathbb{R}^2$ and
  hence is a polyhedron in hyperbolic 3-space.
\end{proof}

\

\

\appendix\section{Some auxiliary results}\label{variation of mean curvature and contact angle}

We record the first variation of mean curvature and the contact angle.

\begin{lemma}
  Let $Y$ be a vector field tangent to $\partial M$ along $\partial \Sigma$
  and $u = \langle Y, N \rangle$, then
  \begin{equation}
    \delta H = - \Delta u - (\ensuremath{\operatorname{Ric}}(N) + |A|^2) u -
    \langle \nabla^{\Sigma} H, Y \rangle, \label{first variation mean
    curvature}
  \end{equation}
  and
  \begin{equation}
    \delta \langle X, N \rangle = - \sin \gamma \tfrac{\partial u}{\partial
    \nu} + (- \cos \gamma A (\nu, \nu) + {A_{\partial M}} (\tilde{\nu}, \tilde{\nu})) u +
    \langle L, \nabla^{\partial \Sigma} \gamma \rangle, \label{first variation
    of contact angle}
  \end{equation}
  where $L$ is a bounded vector field.
\end{lemma}

\begin{proof}
  See {\cite[Appendix A]{li-polyhedron-2020}}.
\end{proof}

\section{Evaluation of hyperbolic mass on polyhedra}\label{eval}
We show in this appendix that a component of the hyperbolic mass can be evaluated on a family of polyhedra enclosed by $Z(\vec{a},s)$ (see Theorem \ref{miao}). 

We say that a three-dimensional manifold $(M, g)$ is an asymptotically
hyperbolic manifold if outside a compact set $M$ is diffeomorphic to the
standard hyperbolic space $(\mathbb{H}^3, b)$ minus a geodesic ball and
\begin{equation*}
  |e|_{b} + | \bar{\nabla} e|_b + | \bar{\nabla} \bar{\nabla}
  e|_{b} = O (\mathrm{e}^{- \tau r}),
\end{equation*}
where $\bar{\nabla}$ is the connection on $\mathbb{H}^3$, $r$ is the distance
function to a fixed point $o$ and $\tau > \tfrac{3}{2}$ and $e = g - b$. We can without loss of generality assume that $M$ is
diffeomorphic to $\mathbb{H}^3$. 
We fix $o$ to be the point $(1,0, 0)$ in the upper half space model
\begin{equation*}
  b = \tfrac{1}{(x^1)^2} ((\mathrm{d} x^1)^2 + (\mathrm{d} x^2)^2 +
  (\mathrm{d} x^3)^2) .
\end{equation*}
Then from {\cite[Chapter A]{benedetti-lectures-1992}},
\begin{equation}
  2 \cosh r = \tfrac{1}{x^1} ((x^1)^2 + (x^2)^2 + (x^3)^2 + 1) .
  \label{distance formula}
\end{equation}

First, we show that by a simple computation, the surface $Z (\vec{a}, s)$ is totally
umbilic.

\begin{lemma}
  \label{umbilicity of Z}The second fundamental form of $Z (\vec{a}, s)$
  computed with respect to the unit normal $x^1 \vec{a}$ is $- a^1$ times of
  its induced metric.
\end{lemma}

\begin{proof}
  Denote $N = x^1 \vec{a}$. The second fundamental form of $Z (\vec{a}, s)$ is
  just $\tfrac{1}{2} L_N b$ restricted to $Z (\vec{a}, s)$, where $L_N b$ is
  the Lie derivative of the metric $b$. So the calculation
  \begin{equation*}
    \tfrac{1}{2} L_N b = \tfrac{1}{2} (x^1 a^i \partial_i) (\tfrac{1}{(x^1)^2}
    \delta) = - a^1 \tfrac{1}{(x^1)^2} \delta = - a^1 b
  \end{equation*}
  gives the assertion.
\end{proof}

The mass integrand for an asymptotically hyperbolic manifold (see
{\cite{chrusciel-mass-2003}}) is
\begin{equation*}
  \mathbb{U}(V)= V\ensuremath{\operatorname{div}}e - V \mathrm{d}
  (\ensuremath{\operatorname{tr}}_b e) +\ensuremath{\operatorname{tr}}_b e
  \mathrm{d} V - e (\bar{\nabla} V, \cdot) .
\end{equation*}
We consider the mass integrand integrated over 
 a family of polyhedra indexed by $q$, where the polyhedra are enclosed by $Z(\vec{a},s)$ (not necessarily the polyhedron given in Definitions \ref{cone} and \ref{prism}). We fix now $V=\tfrac{1}{x^1}$ and we assume that the component
$\mathbf{M} (V)$ of the hyperbolic mass functional satisfies the following
\[ \mathbf{M} (V) = \mathbf{M} (\tfrac{1}{x^1}) =\int_{\partial \boldsymbol{{\Delta}}_q} \mathbb{U}^i(\tfrac{1}{x^1})
   \bar{\nu}_i \mathrm{d} \bar{\sigma} + o (1), \label{mass} \]
where $\bar{\nu}$ is the $b$-normal to the face of $\boldsymbol{{\Delta}}_q$
and $\mathrm{d} \bar{\sigma}$ is the two dimensional volume element.

\

Such a family is easy to find. For example, according to
{\cite{chrusciel-mass-2003}} or {\cite{michel-geometric-2011}}, if each
polyhedron of the family is enclosed by a geodesic sphere and encloses another
geodesic sphere, the radius of each sphere goes to infinity as
$\boldsymbol{{\Delta}}_q$ exhaust the manifold $M$, then such a family provides an
example. Let $E_q$ be the all edges of $\boldsymbol{{\Delta}}_q$, $\alpha$ be the
dihedral angle formed by neighbouring faces along $E_q$. We denote by
$\mathrm{d} v$, $\mathrm{d} \sigma$ and $\mathrm{d} \lambda$ respectively the
three, two and one-dimensional volume elements. We put a bar over a letter to
indicate the quantity is calculated with respect to the background metric
$b$.

\begin{theorem}
  \label{miao} Assuming that every dihedral angle satisfies the bound $\sin
  \bar{\alpha} \geq c > 0$, then the mass $\mathbf{M} (V)$ is
\begin{align}
& \mathbf{M} (V) \nonumber\\
= & - \int_{\partial \boldsymbol{{\Delta}}_q} 2 V (H - \bar{H}) \mathrm{d}
\bar{\sigma} + 2 \int_{E_q} V (\alpha - \bar{\alpha}) \mathrm{d}
\bar{\lambda} \nonumber\\
& + \int_{\partial \boldsymbol{{\Delta}}_q} O (\cosh^{- 2 \tau + 1} r)
\mathrm{d} \bar{\sigma} + \int_{E_q} O (\cosh^{- 2 \tau + 1} r) \mathrm{d}
\bar{\lambda} + o (1) . \label{miao-piubello}
\end{align}
\end{theorem}

\begin{remark}
The theorem is actually valid in any dimension.
\end{remark}

By Lemma \ref{umbilicity of Z}, $\bar{H}$ is constant on each reference face.
The formula \eqref{miao-piubello} first appeared in Miao's article
{\cite{miao-measuring-2020}} in the settings of evaluating ADM mass on cubes
in an asymptotically flat manifold. Later this was generalized to more general
polyhedra in {\cite{miao-mass-2021}}. A special case of the polyhedron family
in \eqref{miao-piubello} is a family of parabolic cubes, Jang and Miao
{\cite{jang-hyperbolic-2021}} showed that the terms evaluated on the edges and
the remainder terms dropped out. In particular, $\mathbf{M} (V)$ can be
evaluated on horospheres only.

\begin{proof}[Proof of Theorem \ref{miao}]
  Let $F$ be a face of the polyhedron,
  $\bar{\nu}$ be the $b$-normal pointing outward of $\boldsymbol{{\Delta}}$,
  we see that
  \begin{equation*}
    \bar{\nu} = x^1 a^i \partial_i
  \end{equation*}
  where $a^i$ are constants and $a$ is a vector of length one under the
  Euclidean metric. The standard hyperbolic metric is conformal to the
  Euclidean metric in the upper half space model, so faces of
  $\boldsymbol{{\Delta}}$ meets at constant angles. Easily,
  \begin{equation}
    \partial_{\bar{\nu}} V = x^1 a^i \partial_i (\tfrac{1}{x^n}) = - a^1
    \tfrac{1}{x^1} = - a^1 V. \label{normal derivative of static potential}
  \end{equation}
  Note that $a^1 \in [- 1, 1]$ and the case of $a^1 = \pm 1$ is used in the
  works of Jang-Miao {\cite{jang-hyperbolic-2021}} and the author
  {\cite{chai-asymptotically-2021}} to evaluate the hyperbolic mass (note that \eqref{normal derivative of static potential} was actually first observed by Guo-Xia \cite{guo-partially-2020}, and it predates {\cite{jang-hyperbolic-2021}} and
  {\cite{chai-asymptotically-2021}}). 
  Each face $F$ is umbilic and the mean curvature is then $\bar{H} = - 2 a^1$ by Lemma \ref{umbilicity of Z}.
  Euclidean spheres are also umbilic in the hyperbolic metric, however, they
  do not satisfy the condition \eqref{normal derivative of static potential}.
  
  We have on each face $F$ of $\boldsymbol{{\Delta}}_q$,
  \begin{equation}
    2 V (H - \bar{H}) = -\mathbb{U}^i \bar{\nu}_i
    -\ensuremath{\operatorname{div}}_F (V X) + O (\mathrm{e}^{- 2 \tau r + r})
    \label{mean curvature expansion}
  \end{equation}
  where $X$ is the vector field dual to the 1-form $e (\bar{\nu}, \cdot)$ with
  respect to the metric $b|_F$. This is an easy consequence of
  {\cite[(2.5)]{jang-hyperbolic-2021}} that
\begin{align*}
& \mathbb{U} (\bar{\nu}) \\
= & 2 V (\bar{H} - H) -\ensuremath{\operatorname{div}}_F (V X) +
[(\ensuremath{\operatorname{tr}}_b e - e (\bar{\nu}, \bar{\nu})) \langle
\mathrm{d} V, \bar{\nu} \rangle - V \langle \bar{A}, e \rangle_b] + O
(\mathrm{e}^{- 2 \tau r + r}) .
\end{align*}
  From Lemma \ref{umbilicity of Z} and \eqref{normal derivative of static potential}, we obtain the desired formula \eqref{mean
  curvature expansion}. From \eqref{mean curvature expansion}, we have that
  \begin{equation*}
    \int_{\partial \boldsymbol{{\Delta}}_q} \mathbb{U}^i \bar{\nu}_i \mathrm{d}
    \bar{\sigma} = \int_{\partial \boldsymbol{{\Delta}}_q} [- 2 V (H - \bar{H})
    -\ensuremath{\operatorname{div}}_{\partial \boldsymbol{{\Delta}}_q} (V X)]
    \mathrm{d} \bar{\sigma} + o (1)
  \end{equation*}
  On each face $F$, using the divergence theorem
  \begin{equation*}
    \int_F \ensuremath{\operatorname{div}}_F (V X) \mathrm{d} \bar{\sigma} =
    \int_{\partial F} V e (\bar{\nu}, \bar{n}) \mathrm{d} \bar{\lambda},
  \end{equation*}
  where $\bar{n}$ is the $b$-normal to $\partial F$ in $F$. On the edge
  $F_A \cap F_B$, the contribution is
  \begin{equation*}
    \int_{F_A \cap F_B} V [e (\bar{\nu}_A, \bar{n}_A) + e (\bar{\nu}_B,
    \bar{n}_B)] \mathrm{d} \bar{\lambda} .
  \end{equation*}
  Let $g_{i j} = g (\partial_i, \partial_j)$, we have that
  \begin{equation*}
    \varepsilon_{i j} := (x^1)^2 e_{i j} = (x^1)^2 g_{i j} - (x^1)^2
    b_{i j} = (x^1)^2 b_{i j} - \delta_{i j} = O (\mathrm{e}^{-
    \tau r})
  \end{equation*}
  The $g$-normal to the face $F_A$ is then expressed as
  \begin{equation*}
    \nu_A = \frac{g^{i j} a_j \partial_j}{\sqrt{g^{k l} a_k a_l}} .
  \end{equation*}
  We write
\begin{align*}
\cos \theta = & g (\nu_A, \nu_B) \\
= & (a_i b_j g^{i j}) (g^{k l} a_k a_l)^{- \tfrac{1}{2}} (g^{p q} a_p
a_q)^{- \tfrac{1}{2}} \\
= & a_i b_j \tfrac{g^{i j}}{(x^1)^2} (\tfrac{g^{k l}}{(x^1)^2} a_k a_l)^{-
\tfrac{1}{2}} (\tfrac{g^{p q}}{(x^1)^2} a_p a_q)^{- \tfrac{1}{2}} .
\end{align*}
  Up to here, it follows from the same lines as in
  {\cite[(3.9)-(3.30)]{miao-mass-2021}} to show that
\begin{align*}
& \int_{F_A \cap F_B} V [e (\bar{\nu}_A, \bar{n}_A) + e (\bar{\nu}_B,
\bar{n}_B)] \mathrm{d} \bar{\lambda} \\
= & \int_{F_A \cap F_B} V (\bar{\alpha} - \alpha) + \int_{F_A \cap F_B} V
\cosh^{- 2 \tau} r \mathrm{d} \bar{\lambda} .
\end{align*}
  Therefore,
\begin{align*}
& \mathbf{M} (V) \\
= & - \int_{\partial \boldsymbol{{\Delta}}_q} 2 V (H - \bar{H}) \mathrm{d}
\bar{\sigma} + 2 \int_{E_q} V (\alpha - \bar{\alpha}) \mathrm{d}
\bar{\lambda} \\
& + \int_{\partial \boldsymbol{{\Delta}}_q} \cosh^{- 2 \tau + 1} r \mathrm{d}
\bar{\sigma} + \int_{E_q} \cosh^{- 2 \tau + 1} r \mathrm{d} \bar{\lambda}
+ o (1),
\end{align*}
  obtaining the theorem.
\end{proof}

\section{Proof of Proposition \ref{prop_energy_est}}\label{proof of minimiser in general tetrahedron}
\

Before giving the proof of Proposition \ref{prop_energy_est}, we give a result about positive-definite matrices that we use later on.

\begin{lemma}
	\label{lem_metric_matrix}
	Given any vector $b=(b_1,b_2,b_3)\in [-1,1]^3$, we define the matrix $G=G_b$ as
	\[
		G_b:=
		\begin{bmatrix}
			1 & b_3 & b_2\\
			b_3& 1 & b_1 \\
			b_2& b_1 & 1 \\
		\end{bmatrix}.
	\]

	Then $G_b$ is positive-definite if and only if $\beta:=(\arccos b_i)_{i=1}^3$ satisfies the following inequalities.
	\begin{enumerate}
		\item $\beta_i+\beta _{i+1}>\beta _{i+2}$ for all $1\le i\le 3$.
		\item $\sum_{i =1}^{3}\beta_i<2\pi$.
	\end{enumerate}

	We call these inequalities the \textit{spherical triangle inequalities}.
\end{lemma}
\begin{proof}[Proof of Proposition \ref{prop_energy_est}]
	"If" part is quite easy since for $\beta \in [0,\pi]^3$ satisfying spherical triangle inequalities, we can find a non-degenerate spherical convex triangle $n_1n_2n_3$ such that $\left|\widefrown{n_in_{i+1}}\right|=\beta _{i+2}$.
	Since $\left\{ n_i \right\}_{i=1}^3$ forms a basis of $\mathbb{R}^3 $, we know the matrix $(g_{ij})$ defined by $g_{ij}=n_i\cdot n_j$ is positive-definite.

	For the "only if" part, since $G$ is positive-definite, we can decompose $G=\Omega \Omega^t$.
	Here, we use small $t$ to denote the transpose of a matrix.
	We write
	\[
		\Omega=
		\begin{bmatrix}
			n_1 \\ n_2 \\ n_3
		\end{bmatrix}.
	\]

	Clearly $n_i \in \mathbb{S}^2$ and $b_i=n_{i+1}\cdot n_{i+2}$.
	Note that $\left\{ n_i \right\}_{i=1}^3$ also forms a basis of $\mathbb{R}^3 $, and $n_1n_2n_3$ would be a non-degenerate spherical convex triangle if we joint $n_i, n_{i+1}$ using minor geodesic arc for each $1\le i\le 3$.
	Hence $\beta$ satisfies the spherical triangle inequalities.
\end{proof}

\begin{remark}
	Note that the spherical triangle inequalities automatically imply $\beta_i \in (0,\pi)$.
	Hence if $G_b$ is positive-definite, we know $-1<b_i<1$ for $1\le i\le 3$.
	\label{rmk_simple_result}
\end{remark}

\begin{proof}[Proof of Proposition \ref{prop_energy_est}]
	For convenience, we relabel $u_i$ such that $u_i=\frac{n_{i+1}\times n_{i+2}}{\left|n_{i+1}\times n_{i+2}\right|}$. 
	Since $\left\{ n_i \right\}_{i=1}^3$ is a basis of $\mathbb{R}^3 $, we can choose $\left\{ n^i \right\}_{i=1}^3$ as the dual basis of $\left\{ n_i \right\}_{i=1}^3$. In other word, we have $u^i \cdot u_j=\delta^i_j$.
	Note that we have $u_i=-\frac{n^i}{\left|n^i\right|}$ and $n_{i+1}\times n_{i+2}=n^i \mathrm{det}(n_1,n_2,n_3)$.
	In particular, we use the following conventional notations for superscript and subscript as
	\[
		\xi_i=\xi \cdot n_i,\quad \xi^i=\xi \cdot n^i,\quad \text{ and }\bar{\xi}_i=\xi \cdot \bar{n}_i,\quad \bar{\xi}^i=\bar{\xi}\cdot \bar{n}^i,
	\]
	where $\left\{ \bar{n}^i \right\}_{i=1}^3$ is the dual basis of $\left\{ \bar{n}_i \right\}_{i=1}^3$.

	Using Lemma \ref{lem_energy_compute}, we have,
	\begin{align}
		E(P)={}&\sum_{i =1}^{3}\left|F_i \cap P\right|(\cos \alpha_i-\cos \bar{\alpha}_i)\nonumber\\
		={}&\frac{1}{2}\sum_{i =1}^{3}\left|\overrightarrow{OA_{i+1}}\times \overrightarrow{OA_{i+2}}\right|(\xi\cdot n_i-\bar{\xi}\cdot \bar{n}_i)\nonumber\\
		={}& \frac{1}{2}\sum_{i =1}^{3} \frac{\left|n ^{i+1}\times n ^{i+2}\right|}{\left( n^{i+1}\cdot \xi \right)
		\left( n^{i+2}\cdot \xi \right)}(\xi_i-\bar{\xi}_i) \nonumber \\
		={}& \frac{1}{2}\sum_{i =1}^{3}\frac{\left|n_i\right|(\xi_i-\bar{\xi}_i)(n^i\cdot \xi)}{\left|\mathrm{det}(n_1,n_2,n_3)\right|\left( n^{1}\cdot \xi \right)
		\left( n^{2}\cdot \xi \right)\left( n^{3}\cdot \xi \right)} \nonumber \\
		={}& \frac{1}{2\left|\mathrm{det}(n_1,n_2,n_3)\right|\xi^1\xi^2\xi^3}\sum_{i =1}^{3}(\xi^i\xi_i-\xi^i\bar{\xi}_i)
	\end{align}

	Now we define two matrices $G$ and $\bar{G}$ by $G=(g_{ij})_{i,j=1}^3, \bar{G}=(\bar{g}_{ij})_{i,j=1}^3$ with $g_{ij}=n_i\cdot n_j, \bar{g}_{ij}=\bar{n}_i\cdot \bar{n}_j$. For convenience, we choose $b=(b_1,b_2,b_3), \bar{b}=(\bar{b}_1,\bar{b}_2,\bar{b}_3)$ such that
	\[
		G=
		\begin{bmatrix}
			1 & b_3 & b_2\\
			b_3& 1 & b_1 \\
			b_2& b_1 & 1.
		\end{bmatrix},\quad 
		\bar{G}=
		\begin{bmatrix}
			1 & \bar{b}_3 & \bar{b}_2\\
			\bar{b}_3& 1 & \bar{b}_1 \\
			\bar{b}_2& \bar{b}_1 & 1.
		\end{bmatrix}.
	\]

	Note that $b_i=-\cos \theta_i, \bar{b}_i=-\cos \bar{\theta}_i$ here.
	Then we know
	\[
		\sum_{i =1}^{3}\xi^i\xi_i-\xi^i \bar{\xi}_i=
		\sum_{i =1}^{3}\xi^i g_{ij}\xi^j-\xi^i \bar{g}_{ij}\bar{\xi}^j.
	\]
	
	Note that $\xi\cdot u_i< 0$ is equivalent to $\xi^i>0$, we can rewrite our problem into the matrix form.

	\begin{proposition}
	Let $b, \bar{b} \in [-1,1]^3$ be two vectors such that $b_i\le \bar{b}_i$ for all $1\le i\le 3$ and there exists at least one $i$ such that $b_i < \bar{b}_i$.
	We also need to assume $G,\bar{G}$ are all positive-definite matrices.
	Then, for any $\bar{\xi}=(\bar{\xi}^1,\bar{\xi}^2,\bar{\xi}^3)\in (0,+\infty)^3$ with $\bar{\xi} \bar{G}\bar{\xi}^t=1$, we can find $\xi \in (0,+\infty)^3$ with $\xi G\xi^t=1$ such that
	\begin{equation}
		1-\xi \bar{G} \bar{\xi}^t=\xi G \xi^t-\xi \bar{G}\bar{\xi}^t<0.
		\label{eq_prop_matrix_form}
	\end{equation}
	\label{prop_est_pri}
	\end{proposition}
	
	Here, we view $\xi, \bar{\xi}$ as 3-dimensional vectors.

	Now, we prove Proposition \ref{prop_est_pri}.
	Note that by the continuity of $1-\xi \bar{G}\xi^t $, we can relax the condition $\xi \in (0,+\infty)^3$ to $\xi \in [0,+\infty)$.

	We choose $\omega=\bar{\xi}\bar{G}G^{-1}$ and $\xi_0=\frac{\omega}{\sqrt{\omega G \omega^t}}$.
	(Here, $\xi_0$ is the minimiser of function $1-\xi \bar{G}\bar{\xi}^t$ under condition $\xi G \xi^t=1$. Hence, $\xi_0$ is the best $\xi$ that we want.)

	\ 

	\noindent\textbf{First Case.} $\omega^i\ge 0$ for all $1\le i\le 3$.

	We only need to show $1-\xi_0 \bar{G}\bar{\xi}^t<0$, which is equivalent to
	\begin{equation}
		\bar{\xi} \bar{G} G^{-1} \bar{G}\bar{\xi}^t>\bar{\xi}\bar{G}\bar{\xi}^t.\label{eq_pf_comp_matrix}
	\end{equation}

	To prove \eqref{eq_pf_comp_matrix}, we write $h=\bar{b}-b$ and $H=\bar{G}-G$.
	Note that we know
	\[
		H=
		\begin{bmatrix}
			0 & h_3 & h_2\\
			h_3& 0 & h_1 \\
			h_2& h_1 & 0
		\end{bmatrix},
	\]
	and $h$ is a non-zero vector with non-negative entries.

	Hence, we find
	\begin{align*}
		 {} & \bar{\xi}\bar{G}G^{-1}\bar{G}\bar{\xi}^t-\bar{\xi}\bar{G}\bar{\xi}^t=\bar{\xi}(G+H)G^{-1}(G+H)\bar{\xi}^t-\bar{\xi}(G+H)\bar{\xi}^t \\
		={} & \bar{\xi}H\bar{\xi}^t+\bar{\xi}H G^{-1}H\bar{\xi}^t.
	\end{align*}

	Since $\bar{\xi}^i>0$ for all $i$ and $h$ is non-zero, we know $\bar{\xi}H\bar{\xi}^t>0$.
	We also note that $G^{-1}$ is positive-definite.
	Therefore, we know $\xi HG^{-1}H\xi^t\ge 0$.

	This finishes the proof of the inequality (\ref{eq_pf_comp_matrix}).

	\ 

	\noindent\textbf{Second Case.} At least one of $\omega^i<0$. 

	Without loss of generality, we assume $\omega^3<0$.

	To choose a suitable vector $\xi_0$, we rewrite the metric $G,\bar{G},H$ into block form as
	\[
		G=
		\begin{bmatrix}
			G_0 & p^t\\
			p & 1 \\
		\end{bmatrix},\quad 
		\bar{G}=
		\begin{bmatrix}
			\bar{G}_0 & \bar{p}^t\\
			\bar{p} & 1
		\end{bmatrix},\quad 
		H=
		\begin{bmatrix}
			H_0 & q^t \\
			q & 0
		\end{bmatrix}.
	\]
	Here, we have used $p=(b_2,b_1)$, $\bar{p}=(\bar{b}_2,\bar{b}_1)$, $q=(h_2,h_1)$.

	Now we choose
	\[
		\bar{\omega}=(\bar{\omega}^1, \bar{\omega}^2):=\bar{\xi}
		\begin{bmatrix}
			\bar{G}_0 \\ \bar{p}
		\end{bmatrix}
		G_0^{-1}.
	\]
	
	We want to show $\xi_0:= \frac{(\bar{\omega},0)}{\sqrt{(\bar{\omega},0)G(\bar{\omega},0)^t}}$ is the $\xi$ such that (\ref{eq_prop_matrix_form}) holds and $\bar{\omega}^i\ge 0$ for each $i=1,2$.
	Note that we can verify that $\xi_0$ is the minimiser of function $1-\xi \bar{G}\bar{\xi}^t$ under conditions $1-\xi G\xi^t=0$ and $\xi^3=0$ using the method of Lagrange multipliers.

	We divide the proof into three steps.

	\ 

	\noindent\textbf{First Step.} We rewrite $\omega^3<0$ into another suitable form.

	We compute $\omega^3$ as follows.
	\begin{align*}
		{} & \omega^3= \bar{\xi}^3+\sum_{i,j=1 }^{3}g^{3i}h_{ij}\bar{\xi}^j=\bar{\xi}^3+\frac{1}{\mathrm{det}(G)}
		(A_{2},A_{1},1-b_3^2)
		\begin{bmatrix}
			H_0 & q^t\\
			q & 0
		\end{bmatrix}
		\begin{bmatrix}
			\bar{\eta}^t\\
			\bar{\xi}^3 \\
		\end{bmatrix}.
	\end{align*}
	Here, we have used $A_{i}=b_{i+1}b_{i+2}-b_i$, $\eta=(\bar{\xi}^1,\bar{\xi}^2)$ for simplicity.
	To further simplify our inequality, we denote $\zeta=(A_{2},A_{1})$.

	Collect the coefficients for $\bar{\xi}^3$ and $\bar{\eta}$, we find $\omega^3<0$ is equivalent to the following inequality.
	\begin{equation}
		\left(\mathrm{det}(G)+\zeta q^t\right)\bar{\xi}^3+(\zeta H_0+(1-b_3^2)q)\bar{\eta}^t<0.
		\label{eq_pf_w3}
	\end{equation}

	We rewrite it as
	\begin{equation}
		-\frac{\zeta q^t \bar{\xi}^3+\zeta H_0 \bar{\eta}^t}{1-b_3^2}>\frac{\mathrm{det}(G)\bar{\xi}^3}{1-b_3^2}+q\bar{\eta}^t
		\label{eq_pf_condition}
	\end{equation}
	for later application.

	\ 

	\noindent\textbf{Second Step.}
	Inequality (\ref{eq_prop_matrix_form}) holds if we choose $\xi=\xi_0$.
	After substituting $\xi_0$ using the definition of $\bar{\omega}$, we find (\ref{eq_prop_matrix_form}) is equivalent to the following inequality,
	\[
		\bar{\xi}
		\begin{bmatrix}
			\bar{G}_0 \\ \bar{p}
		\end{bmatrix}G^{-1}_0
		\begin{bmatrix}
			\bar{G}_0 & \bar{p}^t
		\end{bmatrix}\bar{\xi}^t>1=\bar{\xi}\bar{G}\bar{\xi}^t.
	\]

	We compute
	\begin{align*}
		 E={} & 
		 \begin{bmatrix}
			 \bar{\eta} &\bar{\xi}^3
		 \end{bmatrix}
		\begin{bmatrix}
			G_0+H_0 \\ \bar{p}
		\end{bmatrix}G^{-1}_0
		\begin{bmatrix}
			G_0+H_0 & \bar{p}^t
		\end{bmatrix}
		\begin{bmatrix}
			\bar{\eta}^t\\ \bar{\xi}^3
		\end{bmatrix}-\begin{bmatrix}
			 \bar{\eta} &\bar{\xi}^3
		 \end{bmatrix}
		 \begin{bmatrix}
			 G_0+H_0 & \bar{p}^t\\
			\bar{p} & 1 \\
		 \end{bmatrix}\begin{bmatrix}
			\bar{\eta}^t\\ \bar{\xi}^3
		\end{bmatrix}
		\\
		={} & \bar{\eta}H_0\bar{\eta}^t+\bar{\eta}H_0 G_0^{-1}H_0\bar{\eta}^t+2\bar{\xi}^3
		\bar{p}G^{-1}_0H_0\bar{\eta}^t+\left(\bar{\xi}^3\right)^2(\bar{p}G^{-1}_0\bar{p}^t-1).
	\end{align*}

	Note that
	\[
		\bar{\eta}H_0\bar{\eta}^t+\bar{\eta}H_0G_0^{-1}H_0\bar{\eta}^t\ge 0
	\]
	and it is strictly greater than 0 if $\bar{b}_3>b_3$ by the argument in the first case. Hence, we only need to focus on the term
	\[
		E_2:=2\bar{\xi}^3
		\bar{p}G^{-1}_0H_0\bar{\eta}^t+\left(\bar{\xi}^3\right)^2(\bar{p}G^{-1}_0\bar{p}^t-1).
	\]

	Substitute $\bar{p}$ by $p+q$ and note $pG_0^{-1}=-\frac{\zeta}{1-b_3^2}$, we have
	\begin{align}
		E_2={} & -\frac{2\bar{\xi}^3 \zeta H_0 \bar{\eta}^t}{1-b_3^2} +2\bar{\xi}^3 q G^{-1}_0 H_0\bar{\eta}^t+\left(\bar{\xi}^3\right)^2\left(- \frac{2\zeta q^t}{1-b_3^2}-\frac{\zeta p^t}{1-b_3^2}+qG_0^{-1}q^t-1\right)  \nonumber \\
		>{} & 2\bar{\xi}^3q\bar{\eta}^t+2
		\bar{\xi}^3qG_0^{-1}H_0 \bar{\eta}^t
		+\left(\bar{\xi}^3\right)^2
		\left( \frac{2\mathrm{det}(G)-\zeta p^t}{1-b_3^2} +qG_0^{-1}q^t-1\right).
		\label{eq_pf_E2_term}
	\end{align}
	Here, we have used condition (\ref{eq_pf_condition}) and $\bar{\xi}^3>0$.

	Now, we note that $q+qG_0^{-1}H_0$ can be written down explicitly as
	\begin{align}
		q+qG^{-1}_0H_0={}&(h_2,h_1)+\frac{(h_1h_3-h_2h_3b_3,h_2h_3-h_1h_3b_3)}{1-b_3^2}\nonumber \\
		={}& \frac{(h_1h_3,h_2h_3)}{1-b_3^2}+
		\frac{(h_2(1-b_3\bar{b}_3), h_1(1-b_3 \bar{b}_3))}{1-b_3^2}.\label{eq_pf_qGH}
	\end{align}

	Recall Remark \ref{rmk_simple_result}, we know the vector $q+qG^{-1}_0H_0$ has non-negative entries.
	Hence
	\begin{equation}
		2\bar{\xi}^3q \bar{\eta}^t+2\bar{\xi}^3qG_0^{-1}H_0\bar{\eta}^t\ge 0.
		\label{eq_pf_E21_term}
	\end{equation}

	For the last term, we have
	\begin{align}
		 {} & 2\mathrm{det}(G)-\zeta p^t+(1-b_3^2)qG_0^{-1}q^t -(1-b_3^2)
		=\mathrm{det}(G)+(1-b_3^2)qG_0^{-1}q^t>0.
		\label{eq_pf_E22_term}
	\end{align}
	Here, we have used $\mathrm{det}(G)=1-b_3^2+\zeta p^t$ and $G_0^{-1}$ is a positive-definite matrix.
	Combining inequalities (\ref{eq_pf_E2_term}), (\ref{eq_pf_E21_term}) and (\ref{eq_pf_E22_term}), we get $E_2>0$ and hence $E>0$.

	\ 

	\noindent\textbf{Third Step.} Show that $\bar{\omega}^1\ge 0, \bar{\omega}^2\ge 0$. By the symmetric property of $\bar{\omega}^1$ and $\bar{\omega}^2$, we only need to show $\bar{\omega}^1\ge 0$. 

	By definition of $\bar{\omega}$, we know
	\[
		\bar{\omega}= \bar{\eta}+\bar{\eta}H_0 G_0^{-1}+\frac{\bar{\xi}^3(\bar{b}_2-\bar{b}_1b_3,\bar{b}_1-\bar{b}_2b_3)}{1-b_3^2}.
	\]

	At first, we note $\bar{\eta}+\bar{\eta}H_0G_0^{-1}$ has non-negative entries using similar calculation for (\ref{eq_pf_qGH}).
	If it happens that the term $\bar{b}_2-\bar{b}_1b_3$ is non-negative, we know $\bar{\omega}^1\ge 0$ easily.

	Hence, we focus the case $\bar{b}_2-\bar{b}_1b_3<0$.
	If we write $\beta_i=\arccos b_i$ and $\bar{\beta}_i=\arccos \bar{b}_i$, we can find $\bar{\beta}_1+\bar{\beta}_2>\beta_3$.
	This can be seen by a simple argument.
	If we suppose $\bar{\beta}_2\le \beta_3-\bar{\beta}_1$ on the contrary, since we know $\bar{\beta}_2, \beta_3-\bar{\beta}_1 \in (0,\pi)$, then
	\begin{align*}
		\cos(\bar{\beta}_1)\cos(\beta_3)>{} & \cos(\bar{\beta}_2)\ge \cos(\beta_3-\bar{\beta}_1)= \cos \beta_3 \cos \bar{\beta}_1+\sin \beta_3 \sin \bar{\beta}_1\\
		>{} & \cos \beta_3 \cos \bar{\beta}_1,
	\end{align*}
	since $\bar{\beta}_1,\beta_3 \in (0,\pi)$.
	This is a contradiction.

	Now, we want to use condition (\ref{eq_pf_w3}) to finish the proof.
	At first, we show the coefficient for $\bar{\xi}^3$ is positive.

	\noindent\textbf{Claim.} $\mathrm{det}(G)+\zeta q^t>0$.

	To prove this claim, we write $\tilde{\beta} =(\bar{\beta}_1,\bar{\beta}_2,\beta_3)$. 
	Then, we know $\tilde{\beta}$ satisfies spherical triangle inequalities since $\beta$ and $\bar{\beta}$ satisfy spherical triangle inequalities, and $\beta^i\ge \bar{\beta}^i$ for $1\le i\le 3$.

	By Lemma \ref{lem_metric_matrix}, we know the matrix $\tilde{G}$ defined by
	\[
		\tilde{G}:=
		\begin{bmatrix}
			1 & b_3 & \bar{b}_2\\
			b_3 & 1 & \bar{b}_1 \\
			\bar{b}_2 & \bar{b}_1 & 1 \\
		\end{bmatrix}
	\]
	is positive-definite.
	In particular, we know $\frac{\tilde{G}+G}{2}$ is positive-definite.
	Then $\mathrm{det}(\frac{\tilde{G}+G}{2})>0$ implies 
	\begin{align*}
		0<{}&1-\left( b_1+\frac{h_1}{2} \right)^2-
		\left( b_2+\frac{h_2}{2} \right)^2-b_3^2+2\left( b_1+\frac{h_1}{2} \right)\left( b_2+\frac{h_2}{2} \right)b_3\\
		={}& \mathrm{det}(G)-b_1h_1-b_2h_2+b_1h_2b_3-b_2h_1b_3-\frac{h_1^2}{4}-\frac{h_2^2}{4}+\frac{h_1h_2}{2} \\
		={}& \mathrm{det}(G)+\zeta q^t -\frac{(h_1-h_2)^2}{4}.
	\end{align*}
	
	Hence $\mathrm{det}(G)+\zeta q^t>\frac{(h_1-h_2)^2}{4}\ge 0$ and we finish the proof for the claim.

	We go back to the proof of $\bar{\omega}^1\ge 0$.
	We can write down $\bar{\omega}^1$ explicitly as
	\[
		\bar{\omega}^1=\bar{\xi}^1+\frac{h_3(\bar{\xi}^2-\bar{\xi}^1b_3)+(\bar{b}_2-\bar{b}_1b_3)\bar{\xi}^3}{1-b_3^2}.
	\]

	Using the condition (\ref{eq_pf_w3}), we only need to show
	\[
		F:=(\mathrm{det}(G)+\zeta q^t)\left( \bar{\xi}^1+ \frac{h_3(\bar{\xi}^2-\bar{\xi}^1 b_3)}{1-b_3^2} \right)-\frac{(\bar{b}_2-\bar{b}_1b_3)(\zeta H_0 +(1-b_3^2)q)\bar{\eta}^t}{1-b_3^2}\ge 0.
	\]

	Since $F$ does not contain $\bar{\xi}^3$ term, we can write $F=F_1\bar{\xi}^1+F_2\bar{\xi}^2$ where $F_i$ is the coefficient of $\bar{\xi}^i$ for $i=1,2$.

	For $F_1$, we have
	\begin{align*}
		F_1={}&\mathrm{det}(G)+\zeta q^t-h_2^2+h_1h_2b_3+h_2A_{2}\\
			  &-\frac{(\mathrm{det}(G)+\zeta q^t)h_3b_3+(h_2-h_1b_3-A_{2})h_3A_{1}}{1-b_3^2}.
	\end{align*}

	To simplify $F_1$, we can verify
	\[
		(\mathrm{det}(G)+\zeta q^t)h_3b_3+(h_2-h_1b_3-A_{2})h_3A_{1}=(1-b_3^2)(-A_{3}h_3-b_1h_2h_3).
	\]
	after a quite long calculation.

	Hence, we find
	\begin{align*}
		F_1={}&\mathrm{det}(G)+A_{1}h_1+2A_{2}h_2+A_{3}h_3+b_3h_1h_2+b_1h_2h_3-h_2^2\\
		={}& \mathrm{det}\left(\frac{\hat{G}+\bar{G}}{2}\right)+\frac{h_1^2+h_3^2}{4}-\frac{b_2h_1h_3}{2}-\frac{h_1h_2h_3}{2}\ge \mathrm{det}\left(\frac{\hat{G}+\bar{G}}{2}\right)
	\end{align*}
	by Cauchy-Schwarz inequality and $\left|h_2+b_2\right|=\left|\bar{b}_2\right|<1$. Here, $\hat{G}$ is given by
	\[
		\hat{G}:=
		\begin{bmatrix}
			1 & b_3 & \bar{b}_2\\
			b_3 & 1 & b_1 \\
			\bar{b}_2 & b_1 & 1 \\
		\end{bmatrix}.
	\]

	Now, we need to show $\hat{G}$ is positive-definite.
	This is easy to see by noting $\hat{\beta}:=(\beta_1,\bar{\beta}_2,\beta_3)$ satisfying spherical triangle inequalities since $\beta$ and $\bar{\beta}$ satisfy spherical triangle inequalities, $\bar{\beta}_1+\bar{\beta}_2>\beta_3$, and $\bar{\beta}_i\le \beta_i$.
	Therefore, we know $\hat{G}$, and hence $\frac{\hat{G}+\bar{G}}{2}$ are all positive-definite.
	This implies $\mathrm{det}(\frac{\hat{G}+\bar{G}}{2})>0$. Hence, $F_1>0$.

	For $F_2$, we have
	\begin{align*}
		F_2={}&\frac{(\mathrm{det}(G)+\zeta q^t)h_3+(A_{2}+h_1b_3-h_2)h_3A_{2}}{1-b_3^2}-(\bar{b}_2- \bar{b}_1b_3)h_1\\
		={}& (1-h_1b_1-b_1^2)-(\bar{b}_2- \bar{b}_1b_3)h_1=1-b_1 \bar{b}_1-(\bar{b}_2- \bar{b}_1b_3)h_1> 0.
	\end{align*}
	Here, we have used $b_1\bar{b}_1<1$ and $(\bar{b}_2- \bar{b}_1b_3)<0, h_1\ge 0$.

	In summary, we have $F\ge 0$ and hence $\bar{\omega}^1\ge 0$. So we finish the proof for Proposition \ref{prop_est_pri}.
	
	In the previous proof, we have assumed $\boldsymbol{C}$ has strictly less dihedral angles  than $\bar{\boldsymbol{C}}$.
	We know the inequality (\ref{eq_energy_est}) could be strict.
	If $\boldsymbol{C}$ have the same dihedral angles with $\bar{\boldsymbol{C}}$, then we know $\boldsymbol{C}$ is isometric to $\bar{\boldsymbol{C}}$. So we can choose a suitable $\xi$ with $E_{\bar{P}}(P_{\xi,\boldsymbol{C}})=0$.
\end{proof}

\begin{remark}
	Note that in general, if $\boldsymbol{C}$ and $\bar{\boldsymbol{C}}$ have same dihedral angles, we cannot conclude $\boldsymbol{C}$ is isometric to $\bar{\boldsymbol{C}}$.
 See Section \ref{sub:some_examples}.
\end{remark}

\bibliographystyle{alpha}
\bibliography{references}

\end{document}